\documentclass{svjour3}
\usepackage{amssymb,amsmath,amsfonts,amscd}
\usepackage{algorithm,algorithmic}
\usepackage{epsfig,epstopdf}
\usepackage{graphicx}
\usepackage{makeidx}
\usepackage{multicol,comment,enumitem}
\usepackage{DIMLS}   
\usepackage{color}

\newcommand{\D}{\mathcal{D}}
\renewcommand{\vec}[1]{\mathbf{{#1}}}
\newcommand{\mat}[1]{\mathbf{{#1}}}
\newcommand{\ipar}{\theta}
\newcommand{\dpar}{\boldsymbol{\theta}}
\newcommand{\GM}[2]{\mathcal{N}({#1},{#2})}
\newcommand{\obs}{\vec{d}}
\newcommand{\pH}{\mat{H}_0}

\smartqed

\date{\today}

\title{Randomized Matrix-free Trace and Log-Determinant Estimators\thanks{The third author
acknowledges the support from the XDATA Program of the Defense Advanced
Research Projects Agency (DARPA), administered through Air Force
Research Laboratory contract FA8750-12-C-0323 FA8750-12-C-0323.}}

\author{Arvind K. Saibaba \and Alen Alexanderian \and Ilse C.F. Ipsen}
\institute{Arvind K. Saibaba\at
Department of Mathematics, North Carolina State University, Raleigh, NC 27695-8205, USA\\\email{asaibab@ncsu.edu } \\
\texttt{http://www4.ncsu.edu/{\char'176}asaibab/} 
\and
Alen Alexanderian\at
Department of Mathematics, North Carolina State University, Raleigh, NC 27695-8205, USA \\
\email{alexanderian@ncsu.edu}\\
 \texttt{http://www4.ncsu.edu/{\char'176}aalexan3/}
 \and
Ilse C. F. Ipsen\at
Department of Mathematics, North Carolina State University, Raleigh, NC 27695-8205, USA \\
\email{ipsen@ncsu.edu}\\
 \texttt{http://www4.ncsu.edu/{\char'176}ipsen/}
}

\begin{document}
\maketitle
\begin{abstract}
We present randomized algorithms for estimating the trace and determinant of Hermitian positive
semi-definite matrices. The algorithms are based on subspace iteration, and access the matrix only
through matrix vector products. We analyse the error due to randomization,
for starting guesses whose elements are Gaussian or Rademacher random variables. The analysis
is cleanly separated into a structural (deterministic) part followed by a probabilistic part.  Our absolute bounds for the expectation and concentration of the estimators are non-asymptotic and informative
even for matrices of low dimension. For the trace estimators, we also present asymptotic bounds 
on the number of samples (columns of the starting guess) required to achieve a user-specified relative error. Numerical experiments illustrate the performance of the estimators and the
tightness of the bounds on low-dimensional matrices; and on
a challenging application in uncertainty quantification arising from  Bayesian optimal experimental
design.
\end{abstract}



\section{Introduction}
Computing the trace of high-dimensional matrices is a common problem in
various areas of applied mathematics, such as  evaluation of
uncertainty quantification measures in parameter estimation and inverse problems \cite{HaberHoreshTenorio08,HaberMagnantLuceroEtAl12,AlexanderianPetraStadlerEtAl14,SaibabaKitanidis15}, 
and generalized cross validation (GCV) \cite{Wahba77,Wahba1990,GolubVonMatt97}. 

Our original motivation came from trace and
log-determinant computations of high-dimensional operators in  Bayesian
optimal experimental design (OED) \cite{ChalonerVerdinelli95}. Of particular
interest is OED for Bayesian inverse problems that are constrained by partial
differential equations (PDEs) with high-dimensional parameters. In
Section~\ref{s_uq} we give an example of such a Bayesian inverse
problem and illustrate the evaluation of OED criteria with our algorithms. 

Trace and determinant computations are straightforward if the matrices are  explicitly defined, 
and one has direct access to 
individual matrix entries. The trace is computed as
the sum of the diagonal elements, while the determinant can be computed as the
product of the diagonal elements from a triangular factor \cite[Section 14.6]{Hig02}.
However, if the matrix dimension is large, or explicit access to individual entries is expensive, 
alternative methods are needed.

Here we
focus on computing the trace and log-determinant of implicitly defined matrices, where 
the matrix can be accessed only through matrix vector products.
We present randomized estimators for $\trace(\ma)$ and\footnote{The square
matrix $\mi$ denotes the identity, with ones on the diagonal and zeros
everywhere else.} $\logdet(\mi+\ma)$ for Hermitian, or real symmetric, positive
semi-definite matrices $\ma\in\cnn$.

\subsection{Main features of our estimator}
Our estimators are efficient and easy to implement, as they are  based on
randomized subspace iteration;  and they are  accurate
for many matrices of interest. 
Unlike Monte Carlo estimators, see Section~\ref{ss_rel}, whose variance 
depends on individual matrix entries, our error bounds rely on eigenvalues. 
To this end we need to assume that the
matrix has a well-defined dominant eigenspace, with a large eigenvalue  gap
whose location is known.  Our bounds quantify the effect of the starting guess on
the dominant eigenspace, and are informative even in the non-asymptotic regime,
for matrices of low dimension. 
Our estimators, although biased, can be much more accurate than Monte Carlo estimators.

\subsection{Contributions}
Our paper makes the following four contributions.

\subsubsection{Randomized estimators}  
Assume that the Hermitian positive semi-definite matrix $\ma\in\cnn$ has $k$ dominant 
eigenvalues separated by  a gap
from the remaining $n-k$ sub-dominant eigenvalues,
$\lambda_1\geq \cdots\geq \lambda_k\gg \lambda_{k+1}\geq\cdots\geq \lambda_n$.
The idea is to capture the dominant eigenspace associated with $\lambda_1, \ldots,\lambda_k$ 
via a low-rank approximation $\mt $ of $\ma$.
Our estimators (Section~\ref{s_alg}) for
 $\trace(\mt)\approx\trace(\ma)$ and $\logdet(\mi+\mt)\approx\logdet(\mi+\ma)$
appear to be new. Here
$\mt \equiv \mq^*\ma\mq \in \mathbb{C}^{\ell \times \ell}$ where $k \leq \ell \ll n$. The matrix
$\mq$ approximates the dominant eigenspace of $\ma$, and
is computed from $q$ iterations of subspace iteration applied
to a starting guess $\mom$, followed by the thin QR factorization of $\ma^q\mom$.

\subsubsection{Structural and probabilistic error analysis}\label{s_analysis}
We derive absolute error bounds for $\trace(\mt)$ and $\logdet(\mi+\mt)$,
for  starting guesses that are Gaussian random variables (Section~\ref{s_gr}),
and Rademacher random variables (Section~\ref{s_rm})
The derivations are cleanly separated into a ``structural'' 
(deterministic) part, followed by a probabilistic part.

\paragraph{Structural analysis (Section~\ref{s_struct}).}
These are perturbation bounds that
apply to any matrix $\mom$, be it random or deterministic.
The resulting absolute error bounds for $\trace(\mt)$ and $\logdet(\mi+\mt)$ 
imply that the estimators are accurate if: (1) the starting guess
$\mom$ has a large contribution in the dominant eigenspace;
(2) the eigenvalue gap is large; and (3) the sub-dominant eigenvalues are negligible. 

The novelty of our analysis is the focus on the eigendecomposition of $\ma$. In contrast, as 
discussed in Section~\ref{s_mc}, the analyses of Monte Carlo estimators depend on the 
matrix entries, and do not take into account the  spectral properties of $\ma$.

To understand the contribution of the random starting guess $\mom$,
let the columns of $\mU_1\in\complex^{n\times k}$ represent an orthonormal basis for the dominant eigenspace, while the columns of  $\mU_2\in\complex^{n\times (n-k)}$ represent an orthonormal basis  associated with the $n-k$ sub-dominant eigenvalues. 
The ``projections" of the starting guess on the respective eigenspaces are
are $\mom_1\equiv \mU_1^*\mom\in\complex^{k\times \ell}$ and
$\mom_2\equiv  \mU_2^*\mom\in\complex^{(n-k)\times \ell}$.

The success of $\mt$ in capturing the dominant subspace $\range(\mU_1)$ 
depends on the quantity\footnote{The superscript $\dagger$ denotes the Moore-Penrose inverse.} 
$\normtwo{\mom_2}\normtwo{\mom_1^\dagger}$.

\paragraph{Probabilistic analysis (Section~\ref{s_prob}).}
We bound the projections $\normtwo{\mom_2}$ and $\normtwo{\mom_1^\dagger}$  
for starting guesses $\mom$ that are Gaussian or Rademacher random matrices.

For Gaussian starting guesses, we present bounds for the mean (or expectation),
and concentration about the mean, based on existing bounds for the spectral norms of 
Gaussian random matrices and their pseudo-inverse.
 
For Rademacher starting guesses, we present Chernoff-type concentration inequalities,
and show that $\ell \sim (k+\log n)\log k$ samples are required to
guarantee $\rank(\mom_1) = k$ with high probability.

\subsubsection{Asymptotic efficiency}
One way to quantify the efficiency of a Monte Carlo estimator is a so-called 
$(\epsilon,\delta)$ estimator \cite{avron2011randomized},
which bounds the number of samples required to 
achieve a relative error of at most $\epsilon$ with probability at least $1-\delta$.  
Our asymptotic $(\epsilon,\delta)$ bounds (Theorem~\ref{t_asymp}) show that  
our trace estimator can require significantly fewer samples than Monte Carlo estimators.

\subsubsection{Numerical Experiments} 
Comprehensive numerical experiments corroborate the performance of our
estimators, and illustrate that our error bounds hold even in the
non-asymptotic regime, for matrices of small dimension
(Section~\ref{num_small}).  Motivated by our desire for fast and accurate
estimation of uncertainty measures in Bayesian inverse problems, we present
a challenging application from Bayesian OED (Section~\ref{s_uq}).

\subsection{Related work} \label{ss_rel}
We demonstrate that the novelty of our paper lies in both, the estimators and their analysis.

There are several  popular estimators for 
the trace of an implicit,  Hermitian positive semi-definite matrix $\ma$,  
the simplest one being a Monte Carlo estimator. It requires 
only matrix vector products with $N$ independently generated
 random vectors $\vz_j$ and computes
$$\trace(\ma) \approx \frac{1}{N} \sum_{j=1}^N \vz_j^*\ma\vz_j.$$
The original algorithm, proposed by Hutchinson~\cite{hutchinson1989stochastic},
uses Rademacher random vectors and produces an unbiased estimator.
Unbiased estimators can also be produced with other distributions, such as
Gaussian random vectors, or columns of the identity matrix that are sampled
uniformly with or without
replacement \cite{avron2011randomized,roosta2015improved}, see the detailed comparison 
in Section~\ref{s_mc}. 

Randomized matrix algorithms \cite{HMT09,Mah11} could furnish a
potential alternative for trace estimation.
Low-rank approximations of $\ma$ can be 
efficiently computed with randomized subspace iteration \cite{liberty2007randomized,martinsson2011randomized} or
Nystr\"{o}m methods \cite{gittens2013revisiting},  
and their accuracy is quantified by probabilistic error bounds in the spectral and Frobenius norms. 
Yet we were not able to find error bounds for the corresponding trace estimator in the literature.

Like our estimators, the spectrum-sweeping method \cite[Algorithm 5]{lin2015randomized} 
is based on a randomized low-rank approximation of $\ma$. However, it is designed to compute
the trace of smooth functions of Hermitian matrices
in the context of \textit{density of state} estimations in quantum physics. 
Numerical experiments illustrate that the method can be much faster than Hutchinson's 
estimator, but there is no formal convergence analysis.

A related problem is the trace computation of the matrix inverse. 
One can combine
a Hutchinson estimator $\frac{1}{N}\sum_{j=1}^N \vz_i^*\ma^{-1}\vz_i$ with quadrature rules for approximating the bilinear forms $\vz_i^*\ma^{-1}\vz_i$ \cite{bai1996some,BaiGolub97}. 
For matrices $\ma$ that are sparse, banded, or whose off-diagonal entries decay away from the main 
diagonal, one can use 
a probing method \cite{tang2012probing} to estimate  the diagonal of $\ma^{-1}$ 
with carefully selected vectors that exploit structure and sparsity.

Computation of the log-determinant is required for maximum likelihood estimation in areas like  
machine learning, robotics and spatial statistics  \cite{ZLW2008}.
This can be achieved  by applying a Monte Carlo algorithm
to the log-determinant directly \cite{BarryPace99}, or to an expansion
 \cite{PaceLeSage2004,ZLW2008}.

Alternatively one can combine the identity $\logdet(\ma) = \trace(\log(\ma))$ \cite[Section 3.1.4]{bai1996some}  
with a Monte Carlo estimator for the trace.
Since computation of $\log(\ma)$,  whether with direct or matrix-free methods,
is expensive for large $\ma$, the logarithm can be expanded into a Taylor series
\cite{boutsidis2015randomized,PaceLeSage2004,ZLW2008},
 a Chebyshev polynomial \cite{han2015large}, or a spline
\cite{anitescu2012matrix,chen2011computing}.

\section{Algorithms and main results}
We present the algorithm for randomized subspace iteration (Section~\ref{s_alg}),
followed by the main error bounds for the trace and logdet estimators (Section~\ref{s_main}),
and conclude with a discussion of Monte Carlo estimators (Section~\ref{s_mc}).

\subsection{The Algorithm}\label{s_alg}
We sketch the estimators for $\trace(\ma)$ and  $\logdet(\mi_n+\ma)$, 
for  Hermitian positive semi-definite matrices  $\ma\in\cnn$ with $k$ dominant eigenvalues.
The estimators relinquish the matrix $\ma$ of order $n$ for a
matrices~$\mt$ of smaller dimension $\ell\ll n$ computed with Algorithm~\ref{alg:randsubspace},
so that $\trace(\mt)$ is an estimator for $\trace(\ma)$,
 and $\logdet(\mi_{\ell}+\mt)$ an estimator for $\logdet(\mi_n+\ma)$. 

Algorithm~\ref{alg:randsubspace} is an idealized version of 
randomized subspace iteration. Its starting guess is a
random matrix $\mom$ with $k\leq \ell\ll n$ columns, sampled from a fixed distribution, 
that is then 
subjected to $q$ power iterations with $\ma$. A thin QR decomposition of the resulting product
 $\ma^q\mom$ produces a matrix $\mq$ with orthonormal columns. The 
 output of Algorithm~\ref{alg:randsubspace} is the $\ell\times \ell$ restriction 
 $\mt = \mq^*\ma\mq$ of $\ma$ to $\ispan(\mq)$. 

\begin{algorithm}[!ht]
\begin{algorithmic}[1]
\REQUIRE Hermitian positive semi-definite matrix $\ma\in\mathbb{C}^{n\times n}$ with
target rank $k$,\\
$\qquad$ Number of subspace iterations $q\geq 1$\\
$\qquad$ Starting guess $\mom \in \complex^{n\times \ell}$ with $k\leq \ell\leq n-k$ columns
\ENSURE Matrix $\mt\in\complex^{\ell\times \ell}$ 
\STATE Multiply $\my = \ma^q\mom$ 
\STATE  Thin QR factorization $\my=\mq\mr$
\STATE Compute $\mt = \mq^*\ma\mq$.
\end{algorithmic}
\caption{Randomized subspace iteration (idealized version)}
\label{alg:randsubspace}
\end{algorithm}

The idealized subspace iteration in Algorithm~\ref{alg:randsubspace} can be 
numerically unstable.  The standard remedy is to alternate matrix
products and  QR factorizations \cite[Algorithm 5.2]{saad1992numerical}. 
In practice, one can trade off numerical
stability and efficiency by computing the QR factorization once every few
steps \cite[Algorithm 5.2]{saad1992numerical}. 
Throughout this paper, we assume exact arithmetic and do not take into account finite precision effects.

\paragraph{Random starting guess.} 
The entries of $\mom$ are i.i.d.\footnote{independent and identically distributed} variables 
from one of the two distributions: 
standard normal (zero mean and variance 1), or
Rademacher (values $\pm 1$ with equal probability). 

As in Section~\ref{s_analysis}, 
let $\mom_1\equiv \mU_1^*\mom$ and $\mom_2\equiv  \mU_2^*\mom$ be the respective
``projections" of the starting guess on the dominant and subdominant eigenspaces.
The success of $\mt$ in capturing the dominant subspace depends on the quantity
$\normtwo{\mom_2}\normtwo{\mom_1^\dagger}$.
We make the reasonable assumption $\rank(\mom_1)=k$, so that 
$\mom_1^{\dagger}$ is a right inverse.
Asymptotically, for both Gaussian \cite[Propositions A.2 and A.4]{HMT09}
and Rademacher random matrices \cite[Theorem 1.1]{rudelson2009smallest},
 $\|\mom_2\|_2$ grows
like $\sqrt{n-k} + \sqrt{\ell}$, and  $1/\|\mom_1^{\dagger}\|_2$  like $\sqrt{\ell} - \sqrt{k}$.

Other than that, however, there are major differences.
For Gaussian random matrices, the number columns in $\mom$ is $\ell=k+p$, where $p$
is a  user-specified oversampling parameter.
  The discussion in \cite[Section 5.3]{gu2015subspace}
indicates that the bounds in Section~\ref{s_gr} should hold with high probability for $p\lesssim 20$.
Asymptotically, the required number of columns in a Gaussian starting guess is $\ell\sim k$. 

In contrast, the number of columns in a Rademacher random matrix cannot simply be relegated,
once and for all, to a fixed oversampling parameter, but instead show a strong
dependence on the dimension  $k$ of the dominant subspace and the matrix dimension $n$. We show (Section~\ref{s_prob})
 that the error bounds in Section~\ref{s_rm} hold with high probability,
 if the number of columns in $\mom$ is  $\ell \sim (k+\log n)\log k$. This behavior is similar to that 
 of structured random matrices from sub-sampled random Fourier transforms 
 and  sub-sampled random Hadamard transforms~\cite{tropp2011improved}.  
It is not yet clear, though, whether the asymptotic factor $(k+\log n)\log k$ is tight,
 or whether it is merely an artifact of the analysis.

\subsection{Main results}\label{s_main}
We clarify our assumptions (Section~\ref{s_ass}),
before presenting the main error bounds for the trace and logdet estimators,
when the random matrices for the starting guess are Gaussian (Section~\ref{s_gr}) 
and Rademacher (Section~\ref{s_rm}).

\subsubsection{Assumptions}\label{s_ass}
Let $\ma\in\cnn$ be a Hermitian positive semi-definite matrix with eigenvalue decomposition
$$\ma =\mU\mlam\mU^*, \qquad 
\mlam=\diag\begin{pmatrix} \lambda_1 & \cdots & \lambda_n\end{pmatrix}\in\rnn,$$
where the eigenvector matrix $\mU\in\cnn$ is  unitary, and 
the eigenvalues are ordered, $\lambda_1\geq \cdots \geq \lambda_n\geq 0$.

We assume that the eigenvalues of $\ma$ have  a gap $\lambda_k>\lambda_{k+1}$ 
for some $1\leq k<n$, and
distinguish the dominant eigenvalues from the sub-dominant ones by partitioning
$$\mlam=\begin{pmatrix}\mlam_1 & \\ & \mlam_2 \end{pmatrix}, \qquad
\mU=\begin{pmatrix}\mU_1 & \mU_2\end{pmatrix},$$
where 
$\mlam_1=\diag\begin{pmatrix}\lambda_1 & \cdots & \lambda_k\end{pmatrix}\in\real^{k\times k}$
is nonsingular, and $\mU_1\in\complex^{n\times k}$.
The size of the gap is inversely proportional to
$$\gamma\equiv \lambda_{k+1}/\lambda_k =\|\mlam_2\|_2\,\|\mlam_1^{-1}\|_2<1.$$

Given a number of power iterations  $q\geq 1$, 
and a starting guess $\mom\in\complex^{n\times \ell}$ 
with $k\leq \ell\leq n$ columns, we assume that the product has full column rank,
\begin{eqnarray}\label{e_ass1}
\rank(\ma^q\,\mom)=\ell.
\end{eqnarray}

Extract an orthonormal basis for $\range(\ma^q\,\mom)$ with a thin QR decomposition
$\ma^q\,\mom=\mq\mr$, where $\mq\in\complex^{n\times \ell}$ 
with $\mq^*\mq=\mi_{\ell}$, and the matrix $\mr\in\complex^{\ell\times \ell}$ nonsingular.

To distinguish of the effect of the dominant subspace on the starting guess from that of the 
subdominant space, partition
$$\mU^*\mom=\begin{pmatrix}\mU_1^*\mom \\ \mU_2^*\mom\end{pmatrix}=
\begin{pmatrix}\mom_1 \\ \mom_2 \end{pmatrix}, $$
where $\mom_1\equiv \mU_1^*\mom\in\complex^{k\times \ell}$ and 
$\mom_2\equiv  \mU_2^*\mom\in\complex^{(n-k)\times \ell}$. 
We assume that $\mom$ has a sufficient contribution in the dominant subspace of $\ma$, 
\begin{eqnarray}\label{e_ass2}
\rank(\mom_1)=k.
\end{eqnarray}


\subsubsection{Gaussian random matrices}\label{s_gr}
We present absolute error bounds for the trace and logdet estimators when the random
starting guess $\mom$ in Algorithm~\ref{alg:randsubspace} is a Gaussian. The bounds
come in two flavors: expectation, or mean (Theorem~\ref{t_gauss_exp});
and concentration around the mean 
(Theorem~\ref{t_gauss}). We argue that for matrices with sufficiently dominant eigenvalues, the 
bounds are close.

The number of columns in $\mom$ is equal to 
$$\ell=k+p,$$
where $0\leq p<n-k$ is a user-specified oversampling parameter. 
We abbreviate
\begin{equation}\label{e_mu}
\mu \equiv \sqrt{n-k} + \sqrt{k+p}.
\end{equation}

\begin{theorem}[Expectation]\label{t_gauss_exp}
With the assumptions in Section~\ref{s_ass}, let $\mt$ be computed by 
Algorithm~\ref{alg:randsubspace} with a Gaussian starting guess $\mom^{n\times (k+p)}$ and furthermore, let $p\geq 2$. Then 
\begin{eqnarray*}
0 \leq \expect{\trace(\ma) - \trace( \mt)} \leq 
\left(1 + \gamma^{2q-1}\,C_{ge}\right) \> \trace(\mlam_2), \end{eqnarray*}
and
\begin{eqnarray*}
\lefteqn{0\leq \expect{\logdet( \mi+ \ma) -  \logdet( \mi+  \mt)}  \leq} \\  
& & \qquad\qquad\qquad 
\logdet\left(  \mi+ \mlam_2\right) +
 \logdet\left( \mi+ \gamma^{2q-1}\, C_{ge}\,\mlam_2\right),
\end{eqnarray*}â
where 
$$C_{ge} \define \frac{e^2\,(k+p)}{(p+1)^2}\>
 \left(\frac{1}{2\pi(p+1)}\right)^{\frac{2}{p+1}}\>
 \left(\mu+\sqrt{2} \right)^2 \left(\frac{p+1}{p-1}\right). $$
\end{theorem}

\begin{proof}
See Section~\ref{ss_g_exp}. \qed
\end{proof}

Theorem~\ref{t_gauss_exp} demonstrates that Algorithm~\ref{alg:randsubspace}
with a Gaussian starting guess produces a biased estimator. However, when
$\mlam_2 = \mathbf{0}$, then Algorithm~\ref{alg:randsubspace} produces an
unbiased estimator. 

In the special case when $\rank(\ma)=k$,  the assumption (\ref{e_ass2}) guarantees exact 
computation,
$\trace(\mt)=\trace(\ma)$ and $\logdet(\mt)=\logdet(\ma)$. Hence
the bounds are zero, and hold with equality.
If $\ma$ has $n-k$ eigenvalues close to zero, i.e. $\mlam_2\approx 0$, the
upper bounds in Theorem~\ref{t_gauss_exp} are small, implying that the estimators are accurate 
in the absolute sense.
If $\ma$ has $k$ dominant eigenvalues that are very well separated from the remaining eigenvalues, 
i.e. $\gamma\ll 1$, then Theorem~\ref{t_gauss_exp} implies that the absolute error in the estimators
depends on the mass of the neglected eigenvalues $\mlam_2$. 
The above is true also for the following concentration bounds,
which  have the same form as the expectation bounds. 

\begin{theorem}[Concentration]\label{t_gauss}
With the assumptions in Section~\ref{s_ass}, let $\mt$ 
be computed by Algorithm~\ref{alg:randsubspace} with a Gaussian starting guess 
guess $\mom^{n\times (k+p)}$ where $p \geq 2$.
If $0<\delta<1$, then with probability at least $1-\delta$ 
\begin{eqnarray*} 
0 \leq \trace(\ma) - \trace( \mt) \leq \> \left(1 + \gamma^{2q-1}\,C_g\right) \>\trace(\mlam_2), 
\end{eqnarray*}
and
\begin{eqnarray*}
0 \leq \lefteqn{\logdet( \mi+ \ma) -  \logdet( \mi+  \mt)  \leq}\\
& & \qquad\qquad\qquad  
\logdet\left(  \mi+ \mlam_2\right) 
+ \logdet\left( \mi+ \gamma^{2q-1} \,C_g\,\mlam_2\right),
\end{eqnarray*}
where 
$$C_g  \equiv \frac{e^2\,(k+p)}{(p+1)^2}\> \left(\frac{2}{\delta}\right)^{\frac{2}{p+1}} 
\left(\mu  + \sqrt{2\log\frac{2}{\delta}}\right)^2.$$ 
\end{theorem}

\begin{proof}
Substitute Lemma~\ref{p_gauss_dev} into Theorems~\ref{t_te} and~\ref{t_ld}.~\qed
\end{proof}
\bigskip

The expectation and concentration bounds in Theorems \ref{t_gauss_exp} and~\ref{t_gauss} are the same save for the constants 
$C_{ge}$ and $C_g$. For matrices $\ma$ with sufficiently well separated eigenvalues,
i.e. $\gamma\ll 1$,
 and sufficiently many power iterations $q$ in Algorithm~\ref{alg:randsubspace}, 
 the factor $\gamma^{2q-1}$ subdues the effect of $C_{ge}$ and $C_g$, so that 
Theorems \ref{t_gauss_exp} and~\ref{t_gauss} are effectively the same. 

Nevertheless, we can still compare Theorems 
\ref{t_gauss_exp} and~\ref{t_gauss} by comparing their constants. 
To this end we take advantage of the natural logarithm, and consider two cases.
For a high failure probability $\delta = 2/e$, the ratio is
$$\frac{C_g}{C_{ge}} = \left(2\,e\,\pi\,(p+1)\right)^{\frac{2}{p+1}}\>\left(\frac{p-1}{p+1}\right)
 \rightarrow 1 \qquad \text{as} \quad p\rightarrow \infty.$$ 
Hence  the concentration bound approaches the
expectation bound as the oversampling increases. Note, though, that the rank assumptions
for the bounds impose the limit $p<n-k$. However, for the practical value $p=20$, the ratio
$C_g/C_{ge}\approx 1.6$, so that the constants differ by a factor less than 2.

For a lower failure probability $\delta < 2/e$,  we have $C_g > C_{ge}$. Hence the
concentration bound in Theorem~\ref{t_gauss} has a higher constant.

\subsubsection{Rademacher random matrices}\label{s_rm}
We present absolute error bounds for the trace and logdet estimators when the random
starting guess $\mom$ in Algorithm~\ref{alg:randsubspace} is a Rademacher random matrix. 
In contrast to Gaussian starting guesses, the number of columns in the Rademacher guess reflects
the dimension of the dominant subspace.

The error bounds contain a parameter $0<\rho<1$ that controls the magnitude of 
$\normtwo{\mom_1^\dagger}$. 
The bound below has the same form as the error bound in Theorem~\ref{t_gauss}
with Gaussian starting guesses; the only difference being the constant.

\begin{theorem}\label{t_rademacher}
With the assumptions in Section~\ref{s_ass},  let $0<\delta<1$ be a given failure probability,
and let $\mt$ be computed by Algorithm~\ref{alg:randsubspace} with a Rademacher 
starting guess $\mom\in\real^{n\times \ell}$.
If the number of columns in $\mom$ satisfies 
\[  \ell \geq 2\rho^{-2}\left(\sqrt{k} + \sqrt{8\log\frac{4n}{\delta}}\right)^2 \log\left(\frac{4k}{\delta}\right), \]
then with probability at least $1-\delta$ 
\begin{eqnarray*} 
0 \leq \trace(\ma) - \trace(\mt) \leq \> \left(1 + \gamma^{2q-1}\, C_r\right) \>\trace(\mlam_2), 
\end{eqnarray*}
and
\begin{eqnarray*}
\lefteqn{0\leq \logdet(\mi+\ma) - \logdet( \mi+\mt) \leq}\\
& & \qquad\qquad\qquad  
\logdet\left(\mi+ \mlam_2\right) + 
\logdet\left( \mi+ {\gamma^{2q-1}}\, C_r\,\mlam_2\right),
\end{eqnarray*}
where
\[ C_r \equiv \> \frac{1}{(1-\rho)}\left[1 + 3\ell^{-1}\left(\sqrt{n-k} + \sqrt{8\log\frac{4\ell}{\delta}}\right)^2 \log\frac{4(n-k)}{\delta}\right].\]
\end{theorem}
\begin{proof}
Substitute the bound for $\normtwo{\mom_2}^2\normtwo{\mom_1^\dagger}^2$ from  Theorem~\ref{p_gauss_dev} into Theorems~\ref{t_te} and~\ref{t_ld}.~\qed
\end{proof}

The interpretation of Theorem is the same as that of 
Theorems \ref{t_gauss_exp} and~\ref{t_gauss}. In contrast to Gaussian starting guesses,
whose number of columns depends on  a fixed oversampling parameter $p$, the columns
of the Rademacher guess increase with the dimension of the dominant subspace. 

Theorem~\ref{t_gauss_exp} shows that when Algorithm~\ref{alg:randsubspace} is
run with a Gaussian starting guess, the resulting estimators for the trace and
determinant are biased. We are not able to provide a similar result for the
expectation of the estimators for the Rademacher starting guess. However, we
conjecture that the estimators for trace and determinant are biased even when
the Rademacher starting guess is used in Algorithm~\ref{alg:randsubspace}.

\subsection{Comparison with Monte Carlo estimators}\label{s_mc}

The reliability of Monte Carlo estimators is judged by the variance of a
single  sample. This variance is  $2(\normf{\ma}^2 - \sum_{j=1}^n\ma_{jj}^2)$
for the Hutchinson estimator, and $2\normf{\ma}^2$ for the Gaussian estimator.
Avron and Toledo~\cite{avron2011randomized} were the first to determine
the number of Monte Carlo samples $N$ required to achieve a
\textit{relative} error $\epsilon$ with probability $1-\delta$, and
defined an $(\epsilon,\delta)$ estimator
$$\prob{\left|\trace(\ma)-\frac{1}{N}\sum_{j=1}^N \vz_j^*\ma\vz_j\right| 
\leq \epsilon \>\trace(\ma)} \geq 1 - \delta. $$
An $(\epsilon,\delta)$ estimator based on Gaussian vectors $\vz_j$ requires
$N \geq 20\,\epsilon^{-2}\, \log(2/\delta)$ samples. In contrast, the Hutchinson
estimator, which is based on Rademacher vectors, 
requires $N \geq 6\,\epsilon^{-2} \,\log(2\rank(\ma)/\delta)$ samples.

Roosta-Khorasani and Ascher~\cite{roosta2015improved} improve the above bounds for Gaussian estimators to $N \geq 8\, \epsilon^{-2}\, \log(2/\delta)$; and for the Hutchinson estimator to $N \geq 6\,\epsilon^{-2}\, \log(2/\delta)$, 
thus removing the dependence on the rank. They also derived 
bounds on the number of samples required for an $(\epsilon,\delta)$ estimator, using the Hutchinson, Gaussian and the unit vector random samples, which  depend on specific properties of $\ma$. 
All bounds retain the $\epsilon^{-2}$ factor, though, which means that an accurate
trace estimate requires many samples in practice. In fact, even for small matrices, while a few samples can estimate the trace up to one digit of accuracy, many samples are needed in practice to estimate the trace to machine precision. 

To facilitate comparison between our estimators and the Monte Carlo estimators, we derive the number of iterations needed for an $(\epsilon,\delta)$ estimator. Define the relative error 
\begin{equation}\label{e_Delta} \Delta \equiv \trace(\mlam_2)/\trace(\mlam). \end{equation}  
In practice, the relative error $\Delta$ is not known. Instead, it can be estimated as follows: the bounds  $\trace(\mlam_2) \leq (n-k)\lambda_{k+1}$, $\trace(\mlam_1) \geq k\lambda_k$, can be combined to give us the upper bound 
\[ \Delta \leq \frac{(n-k)\gamma}{n\gamma + k(1-\gamma)} \>. \]
Assuming that $\Delta > 0$, abbreviate $\epsilon_\Delta \equiv \epsilon/\Delta$. If $\Delta = 0$, then we have achieved our desired relative error, i.e., the relative error is less than $\epsilon$. 
 
We present the following theorem that gives the asymptotic bound on the number of matrix-vector products needed for an $(\epsilon,\delta)$ trace estimator. 
\begin{theorem}[Asymptotic bounds]\label{t_asymp}
With the assumptions in Section~\ref{s_ass}, let $\epsilon$ be the desired accuracy and let $0 < \Delta < \epsilon \leq 1$. The number of matrix-vector products for an $(\epsilon,\delta)$ estimator is asymptotically
\begin{equation}
k\left(\log \frac{1}{\epsilon_\Delta - 1} + \log \frac{2}{\delta}\right),
\end{equation}
for Gaussian starting guess, whereas for Rademacher starting guess the number of matrix-vector products is asymptotically
\begin{equation}
(k+\log n)\log k \left(\log \frac{1}{\epsilon_\Delta - 1} + \log \left[(n-k)\log \frac{4n}{\delta}\right]\right).
\end{equation}
\end{theorem}
\begin{proof}
The number of matrix-vector products in Algorithm~\ref{alg:randsubspace} are $\ell (q + 1)$. Recall that the number of samples required for Gaussian starting guess are $\ell \sim k$; whereas for Rademacher starting guess $\ell \sim 
(k+\log n)\log k$. With probability of failure at most $\delta$, for an $(\epsilon,\delta)$ estimator

\[ \frac{\trace(\ma) - \trace(\mt)}{\trace(\ma)} \leq \>  (1+\gamma^{2q-1}C)\Delta. \]
Here $C$ can either take values $C_g$ for standard Gaussian matrices and $C_r$ for standard Rademacher matrices. Equating the right hand side to $\epsilon$ gives us  $(1+\gamma^{2q-1}C) \Delta = \epsilon$. Assuming $\epsilon > \Delta$, we can solve for $q$ to obtain
\[ q =  \left\lceil \frac{1}{2}\left(1 +  \log\left(\frac{C\Delta}{\epsilon - \Delta}\right) \Bigg/ \log \gamma^{-1} \right) \right\rceil. \]
 Asymptotically, $\log C_g$ behaves like $\log 2/\delta$ and $\log C_r$ behaves like $\log 
\left[(n-k)\log {4n}/{\delta}\right]$. This proves the desired result.~\qed
\end{proof}

Theorem~\ref{t_asymp} demonstrates both estimators are computationally efficient compared to the Monte Carlo estimators if $\Delta$ is sufficiently small.

\section{Structural analysis}\label{s_struct}
We defer the probabilistic part of the analysis as long as possible, and start with 
deterministic error bounds for $\trace(\mt)$  (Section~\ref{s_tracebounds}) 
and $\logdet(\mt)$ (Section~\ref{s_logdetbounds}),
where $\mt$ is the restriction of $\ma$ computed by Algorithm~\ref{alg:randsubspace}. 
These deterministic bounds are called ``structural'' because they
hold for \textit{all} matrices $\mom$
that satisfy the rank conditions (\ref{e_ass1}) and (\ref{e_ass2}).

\subsection{Trace Estimator}\label{s_tracebounds}
We derive the following absolute error bounds
for Hermitian positive semi-definite matrices~$\ma$ and matrices $\mt$
computed by Algorithm~\ref{alg:randsubspace}.

\begin{theorem}\label{t_trace}
With the assumptions in Section~\ref{s_ass}, let $\mt=\mq^*\ma\mq$ be computed by 
Algorithm~\ref{alg:randsubspace}. Then
\begin{eqnarray*}
0\leq \trace(\ma)-\trace(\mt)\leq \left(1  + \theta_1\right)\>\trace(\mlam_2)
\end{eqnarray*}
where
$\ \theta_1\equiv \min\{\gamma^{q-1}\,\normtwo{\mom_2\mom_1^{\dagger}},\>
\gamma^{2q-1}\,\|\mom_2\mom_1^{\dagger}\|_2^2\}$.
\end{theorem}

\begin{proof}
The lower bound is derived in Lemma~\ref{l_tel}, and the 
upper bounds in Theorem~\ref{t_te}. \qed
\end{proof}

Theorem~\ref{t_trace} implies that $\trace(\mt)$ has a small absolute error if 
Algorithm~\ref{alg:randsubspace} applies a sufficient number $q$ of power iterations.
More specifically, only a few iterations are required if the eigenvalue gap is large and $\gamma\ll 1$.
The term $\theta_1$ quantifies the contribution of the starting guess $\mom$ in the dominant subspace $\mU_1$. The minimum in $\theta_1$ is attained by
$\gamma^{q-1}\,\normtwo{\mom_2\mom_1^{\dagger}}$  when,
relative to the eigenvalue gap and the iteration count $q$, the starting guess
$\mom$ has only a ``weak'' contribution in the dominant subspace.

We start with the derivation of the lower bound, which relies on the 
variational inequalities for Hermitian matrices, and shows that
the trace of a restriction can never exceed that of the original matrix.

\begin{lemma}\label{l_tel}
With the assumptions in Section~\ref{s_ass}, let $\mt=\mq^*\ma\mq$ be computed by 
Algorithm~\ref{alg:randsubspace}. Then 
$$\trace(\ma)-\trace(\mt)\geq \trace(\mlam_2)-(\lambda_{k+1}+\cdots+\lambda_{\ell})\geq 0.$$
\end{lemma}

\begin{proof}
Choose  $\mq_{\perp}\in\complex^{n\times (n-\ell)}$ so that 
$\hat{\mq}\equiv\begin{pmatrix}\mq & \mq_{\perp}\end{pmatrix}\in\cnn$
is unitary, and partition
$$\hat{\mq}^*\ma\hat{\mq}
=\begin{pmatrix}\mt & \ma_{12}\\ \ma_{12}^* & \ma_{22}\end{pmatrix},$$
where $\mt=\mq^* \ma \mq$ is the important submatrix.
The matrices  $\hat{\mq}^*\ma\hat{\mq}$ and $\ma$ have the same eigenvalues
$0\leq \lambda_n\leq \cdots\leq \lambda_1$. With
$\lambda_{\ell}(\mt)\leq \cdots\leq \lambda_1(\mt)$ being 
the eigenvalues of $\mt$, the Cauchy-interlace theorem \cite[Section 10-1]{Par80} implies
$$0\leq\lambda_{(n-\ell)+j}\leq \lambda_j(\mt)\leq \lambda_j, \qquad 1\leq j\leq \ell.$$
Since $\lambda_j\geq 0$, this implies
(for $\ell=k$ we interpret $\sum_{j=k+1}^{\ell}{\lambda_{j}}=0$)
\begin{eqnarray*} 
\trace(\mt)&\leq &\sum_{j=1}^{\ell}{\lambda_j}=
\trace(\mlam_1) + \sum_{j=k+1}^{\ell}{\lambda_j}\\
&=&\trace(\ma)-\trace(\mlam_2)+\sum_{j=k+1}^{\ell}{\lambda_{j}}\leq \trace(\ma),
\end{eqnarray*}
where the last inequality follows from 
$\sum_{j=k+1}^{\ell}{\lambda_j}\leq \sum_{j=k+1}^n{\lambda_j}=\trace(\mlam_2)$. \qed
\end{proof}
\medskip

Next we derive the two upper bounds.
The first one, (\ref{e_te1}), is preferable when,
relative to the eigenvalue gap and the iteration count $q$, the starting guess
$\mom$ has only a ``weak'' contribution in the dominant subspace.

\begin{theorem}\label{t_te}
With the assumptions in Section~\ref{s_ass}, let $\mt=\mq^*\ma\mq$ be computed by 
Algorithm~\ref{alg:randsubspace}. Then
\begin{eqnarray}\label{e_te1}
\trace(\ma) - \trace(\mt) \leq  
\left(1  + \gamma^{q-1}\>\normtwo{\mom_2\mom_1^{\dagger}}\right)\>\trace(\mlam_2).
\end{eqnarray}
If $0<\|\mom_2\mom_1^{\dagger}\|_2\leq \gamma^{-q}$, then the following bound is tighter,
\begin{eqnarray}\label{e_te2}
\trace(\ma)-\trace(\mt)\leq 
\left(1+\gamma^{2q-1}\>\|\mom_2\mom_1^{\dagger}\|_2^2\right)\> \trace(\mlam_2).
\end{eqnarray}
\end{theorem}

\begin{proof}
The proof proceeds in six steps. The first five steps are the same for both bounds.

\paragraph{1. Shrinking the space from $\ell$ to $k$ dimensions}
If $\mw\in\complex^{\ell\times k}$ is any matrix with orthonormal columns, then
Lemma~\ref{l_tel} implies
$$ \trace(\ma)-\trace(\mq^*\ma\mq)\leq \,
\mathcal{U}\equiv\trace(\ma)-\trace\left((\mq\mw)^*\,\ma\,(\mq\mw)\right).$$
The upper bound $\mathcal{U}$ replaces the matrix $\mq^*\ma\mq$ of order $\ell$ 
by the matrix $(\mq\mw)^*\,\ma\,(\mq\mw)$ of order $k\leq\ell$. 
The eigendecomposition of $\ma$ yields
$$\trace\left((\mq\mw)^*\,\ma\,(\mq\mw)\right)=t_1+t_2,$$
where  dominant eigenvalues are distinguished from subdominant ones by  
\begin{eqnarray*}
t_1\equiv \trace\left((\mU_1^*\mq\mw)^*\,\mlam_1\,(\mU_1^*\mq\mw)\right), 
\quad 
 t_2\equiv \trace\left((\mU_2^*\mq\mw)^*\,\mlam_2\,(\mU_2^*\mq\mw\right).
 \end{eqnarray*}
Note that $t_1$ and $t_2$ are real. Now we can write the upper bound as
\begin{eqnarray}\label{e_mathcalu}
\mathcal{U}=\trace(\ma)-t_1-t_2.
\end{eqnarray}

\paragraph{2. Exploiting the structure of $\mq$}
Assumption (\ref{e_ass1}) implies that $\mr$ is nonsingular, hence we can solve for $\mq$
in $\ma^q\,\mom=\mq\mr$, to obtain
\begin{eqnarray}\label{e_q}
\mq=(\ma^q\mom)\,\mr^{-1}=\mU\mlam^q\,\mU^*\,\mom\mr^{-1}=\mU\,
\begin{pmatrix}\mlam_1^q\,\mom_1 \\ \mlam_2^q\,\mom_2\end{pmatrix}\,\mr^{-1}.
\end{eqnarray}

\paragraph{3. Choosing $\mw$}
Assumption (\ref{e_ass2}) implies that the $k\times \ell$ matrix $\mom_1$ has full row rank, 
and a right inverse $\mom_1^{\dagger} =\mom_1^*\,(\mom_1\mom_1^*)^{-1}$. 
Our choice for  $\mw$ is 
$$\mw\equiv \mr\,\mom_1^{\dagger}\mlam_1^{-q}\,(\mi_k +\mf^*\mf)^{-1/2} \qquad \text{where}\qquad
\mf\equiv \mlam_2^q\,\mom_2\mom_1^{\dagger}\, \mlam_1^{-q},$$
so that we can express (\ref{e_q}) as 
\begin{eqnarray}\label{e_qw}
\mq\mw=\mU\,\begin{pmatrix}\mlam_1^q\,\mom_1 \\ \mlam_2^q\,\mom_2\end{pmatrix}\,\mr^{-1}\,\mw=
\mU\,\begin{pmatrix}\mi_k\\ \mf\end{pmatrix} (\mi_k +\mf^*\mf)^{-1/2}.
\end{eqnarray}
The rightmost expression shows that $\mq\mw$ has orthonormal columns. 
To see that $\mw$ itself also has orthonormal columns, show that $\mw^*\mw=\mi_k$ 
with the help of 
$$\mr^*\mr=(\mq\mr)^*(\mq\mr)=(\mlam_1^q\,\mom_1)^*(\mlam_1^q\,\mom_1)+
(\mlam_2^q\,\mom_2)^*(\mlam_2^q\,\mom_2).$$

\paragraph{4. Determining $\mathcal{U}$ in (\ref{e_mathcalu})}
From $\mU_1^*\mq\mw=(\mi+\mf^*\mf)^{-1/2}$ in (\ref{e_qw}) follows
$$t_1=\trace\left((\mi+\mf^*\mf)^{-1/2}\,\mlam_1\,(\mi+\mf^*\mf)^{-1/2}\right)=
\trace\left(\mlam_1\, (\mi+\mf^*\mf)^{-1}\right).$$
From (\ref{e_qw}) also follows $\mU_2^*\mq\mw=\mf\,(\mi+\mf^*\mf)^{-1/2}$, so that 
$$t_2=\trace\left((\mi+\mf^*\mf)^{-1/2}\,\mf^*\,\mlam_2\,\mf\,(\mi+\mf^*\mf)^{-1/2}\right)=
\trace\left(\mlam_2\,\mf\, (\mi+\mf^*\mf)^{-1}\,\mf^*\right).$$
Distinguish dominant from subdominant eigenvalues in $\mathcal{U}$ via
$\mathcal{U}=\mathcal{U}_1+\mathcal{U}_2$, where
\begin{eqnarray*}
\mathcal{U}_1\equiv \trace(\mlam_1)-t_1, \qquad \mathcal{U}_2\equiv \trace(\mlam_2)-t_2.
\end{eqnarray*}
Since $t_1$ and $t_2$ are real, so are $\mathcal{U}_1$ and $\mathcal{U}_2$.
With the identity 
$$\mlam_1\,\left(\mi-(\mi+\mf^*\mf)^{-1}\right)=
\mlam_1\,\mf^*\mf\,(\mi+\mf^*\mf)^{-1},$$
and remembering that $\mf=\mlam_2^q\,\mz\,\mlam_1^{-q}$  with
$\mz\equiv\mom_2\mom_1^{\dagger}$, we obtain
$\mathcal{U}=\mathcal{U}_1+\mathcal{U}_2$ with
\begin{eqnarray*}
\mathcal{U}_1 &= &\trace\left(\mlam_1^{1-q}\,\mz^*\mlam_2^q\>\mf\,(\mi+\mf^*\mf)^{-1}\right)\\
\mathcal{U}_2&=&\trace\left(\mlam_2\>(\mi-\mf\,(\mi+\mf^*\mf)^{-1}\mf^*)\right).
\end{eqnarray*}

\paragraph{5. Bounding  $\mathcal{U}$}
Since $\mlam_2$ and $\mi-\mf\,(\mi+\mf^*\mf)^{-1}\mf^*$ both have dimension 
$(n-k)\times (n-k)$,
the \textit{von Neumann trace theorem} \cite[Theorem 7.4.11]{HoJ13} can be applied,
$$\mathcal{U}_2\leq \sum_j{\sigma_j(\mlam_2)\,\sigma_j(\mi-\mf\,(\mi+\mf^*\mf)^{-1}\mf^*)}
\leq \sum_j{\sigma_j(\mlam_2)}=\trace(\mlam_2).$$
The last equality is true because the singular values of a Hermitian positive semi-definite matrix
are also the eigenvalues.
Analogously, $\mlam_2^q\,\mz\,\mlam_1^{1-q}$ and $\mf\,(\mi+\mf^*\mf)^{-1}$ both 
have dimension $(n-k)\times k$, so that
\begin{eqnarray*}
\mathcal{U}_1&\leq& 
\sum_j{\sigma_j(\mlam_2^q\,\mz\,\mlam_1^{1-q})\,\sigma_j(\mf\,(\mi+\mf^*\mf)^{-1})}\\
&\leq & \|\mf\,(\mi+\mf^*\mf)^{-1})\|_2\,\sum_j{\sigma_j(\mlam_2^q\,\mz\,\mlam_1^{1-q})}
\end{eqnarray*}
Repeated applications of the singular value inequalities \cite[Theorem 3.314]{HoJ91} 
for the second factor yield
\begin{eqnarray*}
\sum_j{\sigma_j(\mlam_2^q\,\mz\,\mlam_1^{1-q})}
&\leq &\|\mz\|_2\,\|\mlam_1\|_2^{1-q}\,\sum_j{\sigma_j(\mlam_2^q)}\\
&\leq& \|\mz\|_2\,\|\mlam_1\|_2^{1-q}\,\|\mlam_2\|_2^{q-1}\,\sum_j{\sigma_j(\mlam_2)}
=\gamma^{q-1}\,\|\mz\|_2\,\trace(\mlam_2)
\end{eqnarray*}
Substituting this into the bound for $\mathcal{U}_1$ gives
\begin{eqnarray*}
\mathcal{U}_1&\leq& 
 \|\mf\,(\mi+\mf^*\mf)^{-1}\|_2\>\gamma^{q-1}\,\|\mz\|_2\,\trace(\mlam_2).
\end{eqnarray*}

\paragraph{6. Bounding $\|\mf\,(\mi+\mf^*\mf)^{-1}\|_2$}
For (\ref{e_te1}) we bound $\|\mf\,(\mi+\mf^*\mf)^{-1}\|_2\leq 1$, which yields
$\mathcal{U}_1\leq \gamma^{q-1}\,\|\mz\|_2\,\trace(\mlam_2)$.
For (\ref{e_te2}) we use 
$$\|\mf\,(\mi+\mf^*\mf)^{-1}\|_2\leq 
 \|\mf\|_2\leq \|\mlam_2\|_2^q\,\|\mz\|_2\,\|\mlam_1\|_2^{-q}=\gamma^q\,\|\mz\|_2,$$
which yields $\mathcal{U}_1\leq \gamma^{2q-1}\|\mz\|_2^2\>\trace(\mlam_2)$.

Comparing the two preceding bounds for $\mathcal{U}_1$ shows that
(\ref{e_te2}) is tighter than (\ref{e_te1}) if 
$\gamma^{2q-1}\>\|\mz\|_2^2\leq \gamma^{q-1}\>\|\mz\|_2$,
that is $\|\mz\|_2\leq \gamma^{-q}$. \qed
\end{proof}
\medskip

\begin{remark}\label{r_te}
The two special cases below illustrate that, even in a best-case scenario,
the accuracy of $\trace(\mt)$ is limited by $\trace(\mlam_2)$.

\begin{itemize}
\item If $\ell=k$ and $\mom=\mU_1$ then
$$\trace(\ma) - \trace(\mt) = \trace(\mlam_2).$$
This follows from Lemma~\ref{l_tel}, and from both bounds in Theorem~\ref{t_te}
with $\mom_1=\mi_k$ and $\mom_2=\vzero$.
 \item If $\ell>k$ and $\mom$ consists of the columns of $\mU$ associated
with the dominant eigenvalues $\lambda_1, \ldots,\lambda_{\ell}$ of $\ma$, then
$$\trace(\mlam_2) - (\lambda_{k+1} +\cdots+\lambda_{\ell})\leq 
\trace(\ma) - \trace(\mt) \leq  \trace(\mlam_2).$$
This follows from Lemma~\ref{l_tel}, and from both bounds in Theorem~\ref{t_te}
with $\mom_1=\begin{pmatrix}\mi_k & \vzero_{k\times (k-\ell)}\end{pmatrix}$ and
$\mom_2=\begin{pmatrix} \vzero_{(n-k)\times k} & *\end{pmatrix}$.

Theorem~\ref{t_te} cannot be tight for $\ell>k$ because step~3 of the proof deliberately
transitions to a matrix with $k$ columns. Hence the eigenvalues $\lambda_{k+1},\ldots,\lambda_{\ell}$
do not appear in the bounds of Theorem~\ref{t_te}.
\end{itemize}
\end{remark}

\subsection{Log Determinant Estimator}\label{s_logdetbounds}
Subject to the Assumptions in Section~\ref{s_ass},
we derive the following absolute error bounds 
for Hermitian positive semi-definite matrices~$\ma$ and matrices $\mt$
computed by Algorithm~\ref{alg:randsubspace}.

\begin{theorem}\label{t_logdet}
With the assumptions in Section~\ref{s_ass}, 
let $\mt=\mq^*\ma\mq$ be computed by Algorithm~\ref{alg:randsubspace}. Then
\begin{eqnarray*}
0\leq \log{\det(\mi_n+\ma)}-\log{\det(\mi_{\ell}+\mt)}\leq 
 \log{\det\left(\mi_{n-k}+\mlam_2\right)}+ \log{\det\left(\mi_{n-k}+\theta_2\, \mlam_2\right)}
 \end{eqnarray*}
where
$\ \theta_2\equiv \gamma^{2q-1}\>\normtwo{\mom_2\mom_1^\dagger}^2\,
\min\{1,\tfrac{1}{\lambda_k}\}$.
\end{theorem}

\begin{proof}
The lower bound is derived in Lemma~\ref{l_ldel}, and the 
upper bounds in Theorem~\ref{t_ld}.~\qed 
\end{proof}
\medskip

Theorem~\ref{t_logdet} implies that $\logdet(\mi_{\ell}+\mt)$ has a small absolute error if 
Algorithm~\ref{alg:randsubspace} applies a sufficient number $q$ of power iterations.
As in Theorem~\ref{t_te},
only a few iterations are required if the eigenvalue gap is large and $\gamma\ll 1$.
The term $\theta_2$ quantifies the contribution of the starting guess $\mom$ in the dominant subspace $\mU_1$. The two alternatives differ by a factor of only $\lambda_k^{-1}$. The second one
is smaller if $\lambda_k>1$.

Theorem~\ref{t_ld2} extends Theorem~\ref{t_logdet} to $\log\det(\ma)$ for positive definite $\ma$.

As before, we start with the derivation of the lower bound, which is the counter part of 
Lemma~\ref{l_tel} and shows that the log
determinant of the restriction can never exceed that of the original matrix.

\begin{lemma}\label{l_ldel}
With the assumptions in Section~\ref{s_ass}, let $\mt=\mq^*\ma\mq$ be computed by 
Algorithm~\ref{alg:randsubspace}. Then
\begin{equation*}
\logdet(\mi_n+\ma)- \log{\det\left(\mi_{\ell}+\mt\right)}\geq 
\log{\det(\mi_{n-k}+\mlam_2)}- \log{\prod_{j=k+1}^{\ell}{(1+\lambda_j)}}\geq 0.
\end{equation*}
\end{lemma}

\begin{proof}
Choose the unitary matrix $\hat{\mq}$ as in the proof of Lemma~\ref{l_tel},
$$\hat{\mq}^*(\mi_n+\ma)\hat{\mq}
=\begin{pmatrix}\mi_{\ell}+\mt& \ma_{12}\\ \ma_{12}^* & \mi_{n-\ell}+\ma_{22}\end{pmatrix},$$
and proceed likewise with the Cauchy-interlace theorem \cite[Section 10-1]{Par80}  
to conclude
\begin{eqnarray*} 
\det(\mi_{\ell}+\mt)&\leq &\det(\mi_k+\mlam_1)\prod_{j=k+1}^{\ell}{(1+\lambda_j)}\\
&=&\det(\mi_n+\ma)\prod_{j=k+1}^{\ell}{(1+\lambda_{j})}/\det(\mi_{n-k}+\mlam_2)\leq \det(\mi_n+\ma),
\end{eqnarray*}
where for $\ell=k$ we interpret $\prod_{j=k+1}^{\ell}{(1+\lambda_j)}=1$.
The monotonicity of the logarithm implies
\begin{eqnarray*}
\log{\det(\mi_{\ell}+\mt)}&\leq& \log{\det(\mi_n+\ma)}- 
\log{\det(\mi_{n-k}+\mlam_2)}+\log{\prod_{j=k+1}^{\ell}{(1+\lambda_j)}}\\
&\leq &\log{\det(\mi_n+\ma)}.~\qed
\end{eqnarray*}
\end{proof}
\medskip

The following auxiliary result, often called  \textit{Sylvester's determinant identity},
is required for the derivation of the upper bound.

\begin{lemma}[Corollary 2.1 in \cite{Ou81}]\label{l_syl}
If $\mb\in\cmn$ and $\mc\in\complex^{n\times m}$ then
$$\det(\mi_m\pm\mb\mc)=\det(\mi_n\pm\mc\mb).$$
\end{lemma}

Next we derive two upper bounds. The second one,  (\ref{e_ld2}), is preferable 
when $\lambda_k>1$ because it reduces the extraneous term.

\begin{theorem}\label{t_ld}
With the assumptions in Section~\ref{s_ass}, 
let $\mt=\mq^*\ma\mq$ be computed by Algorithm~\ref{alg:randsubspace}. Then
\begin{eqnarray}\label{e_ld1}
\lefteqn{\log{\det(\mi_n+\ma)}-\log{\det(\mi_\ell+\mt)}\leq}\nonumber\\
&& \qquad\qquad
 \log{\det\left(\mi_{n-k}+\mlam_2\right)}+ \log{\det\left(\mi_{n-k}+\gamma^{2q-1} \,
\normtwo{\mom_2\mom_1^\dagger}^2\>\mlam_2\right)}.
\end{eqnarray}
If $\lambda_k>1$ then the following bound is tighter
\begin{eqnarray}\label{e_ld2}
\lefteqn{\log{\det(\mi_n+\ma)}-\log{\det(\mi_\ell+\mt)}\leq}\nonumber\\
&& \qquad\qquad
 \log{\det\left(\mi_{n-k}+\mlam_2\right)}+ \log{\det\left(\mi_{n-k}+\tfrac{\gamma^{2q-1}}{\lambda_k} \>
\normtwo{\mom_2\mom_1^\dagger}^2\>\mlam_2\right)}.
\end{eqnarray}
\end{theorem}

\begin{proof}
The structure of the proof is analogous to that of Theorem~\ref{t_te},
and the first three steps are the same for (\ref{e_ld1}) and (\ref{e_ld2}).
Abbreviate $f(\cdot)\equiv \log{\det(\cdot)}$.

\paragraph{1. Shrinking the space}
Lemma~\ref{l_ldel} implies
$$f(\mi_n+\ma)-f(\mi_{\ell}+\mq^*\ma\mq)\leq   f(\mi_n+\ma) -f(\mi_k+\mh),$$
where
$$\mh\equiv \mw^*\mt\mw=\mw^*\,\mq^*\ma\mq\,\mw=
(\mU^*\mq\mw)^*\,\begin{pmatrix}\mlam_1 & \\ & \mlam_2\end{pmatrix}\,(\mU^*\mq\mw).$$
The upper bound for the absolute error equals
\begin{eqnarray*}
f(\mi_n+\ma) -f(\mi_k+\mh) & = &f(\mi_k+\mlam_1)+f(\mi_{n-k}+\mlam_2)-f(\mi_k+\mh)\\
&=&f(\mi_{n-k}+\mlam_2) +\mathcal{E}.
\end{eqnarray*}
Since nothing can be done about $f(\mi_{n-k}+\mlam_2)$, it suffices to bound
\begin{eqnarray}\label{e_calE}
\mathcal{E}\equiv f(\mi_k+\mlam_1) -f(\mi_k+\mh).
\end{eqnarray}

\paragraph{2. Exploiting the structure of $\mq$ and choosing $\mw$}
To simplify the expression for $\mh$,
we exploit the structure of $\mq$ and choose $\mw$ as in the proof of Theorem~\ref{t_te}. 
From (\ref{e_qw}) follows
$$\mU^*\mq\mw=\begin{pmatrix}\mi_k\\ \mf\end{pmatrix}(\mi_k+\mf^*\mf)^{-1/2}, \qquad
\text{where}\quad
\mf\equiv \mlam_2^q\,\mom_2\mom_1^{\dagger}\, \mlam_1^{-q}.$$
Substituting this into the eigendecomposition of $\mh$ gives 
\begin{eqnarray*}
\mh=(\mi_k+\mf^*\mf)^{-1/2}\,\left(\mlam_1+\mf^*\mlam_2\mf\right)\,(\mi_k+\mf^*\mf)^{-1/2}.
\end{eqnarray*}

\paragraph{3. Lower bound for $f(\mi_k+\mh)$}
We use the Loewner partial order \cite[Definition 7.7.1]{HoJ13} to represent
positive semi-definiteness, $\mf^*\mlam_2\mf\succeq\vzero$, which implies
\begin{eqnarray}\label{e_H}
\mh \>\succeq \>\mh_1\equiv (\mi_k+\mf^*\mf)^{-1/2}\,\mlam_1 \,(\mi_k+\mf^*\mf)^{-1/2}.
\end{eqnarray}
The properties of the Loewner partial order \cite[Corollary 7.7.4]{HoJ13} imply
$$f(\mi_k+\mh) \geq f(\mi_k+\mh_1).$$
We first derive (\ref{e_ld1}) and then  (\ref{e_ld2}).

\paragraph{Derivation of (\ref{e_ld1}) in Steps 4a--6a}
\begin{description}
\item\textit{4a. Sylvester's determinant identity.}
Applying Lemma~\ref{l_syl} to $\mh_1$ in (\ref{e_H}) gives
$$f(\mi_k+\mh_1)=f(\mi_k+\mh_2) \qquad \text{where} \quad
\mh_2\equiv\mlam^{1/2}_1(\mi_k+\mf^*\mf)^{-1}\mlam_1^{1/2}.$$

\item\textit{5a. Upper bound for $\mathcal{E}$ in (\ref{e_calE}).}
Steps 3 and~4a imply
$$\mathcal{E}\leq f(\mi_k+\mlam_1) - f(\mi_k+\mh_2) =f(\mh_3),$$
where
\begin{eqnarray*}
\mh_3&\equiv& (\mi_k+\mh_2)^{-1/2}(\mi_k+\mlam_1)(\mi_k+\mh_2)^{-1/2}\\
&=&(\mi_k+\mh_2)^{-1}+ (\mi_k+\mh_2)^{-1/2}\mlam_1(\mi_k+\mh_2)^{-1/2}.
\end{eqnarray*}
Expanding the first summand into
$(\mi_k+\mh_2)^{-1}=\mi-(\mi_k+\mh_2)^{-1/2}\mh_2(\mi_k+\mh_2)^{-1/2}$ gives
$$\mh_3=\mi_k+(\mi_k+\mh_2)^{-1/2}(\mlam_1-\mh_2)(\mi_k+\mh_2)^{-1/2}.$$
Since $\mi_k-(\mi_k+\mf^*\mf)^{-1} \preceq \mf^*\mf$, the center term can be bounded by
$$\mlam_1 -\mh_2  = \mlam_1^{1/2}\,\left(\mi_k-(\mi_k+\mf^*\mf)^{-1}\right)\,\mlam_1^{1/2} 
\>\preceq\> \mk\equiv \mlam_1^{1/2} \mf^*\mf\mlam_1^{1/2}.$$
Because the singular values of $(\mi_k+\mh_2)^{-1/2}$ are less than $1$, 
Ostrowski's Theorem  \cite[Theorem 4.5.9]{HoJ13} implies 
$$\mh_3\>\preceq\>\mi_k+(\mi_k+\mh_2)^{-1/2}\mk(\mi_k+\mh_2)^{-1/2} \preceq  \mi_k+\mk.$$
Thus $\mathcal{E} \leq  f(\mh_3) \> \leq\>  f(\mi_k + \mk)$.
\smallskip

\item\textit{6a. Bounding $f(\mi_k+\mk)$.}
Abbreviate $\mg_1 \equiv \mlam_2^{q-1/2}\mom_2\mom_1^{\dagger}\mlam_1^{-q+1/2}$
and expand $\mk$,
$$\mi_k+\mk=\mi_k+\mlam_1^{1/2}\mf^*\mf\mlam_1^{1/2} = \mi_k+\mg_1^*\mlam_2\mg_1.$$
Applying  Lemma~\ref{l_syl},
$$\det(\mi_k+\mg_1^*\mlam_2\mg_1) =\det(\mi_{n-k}+\mg_1\mg_1^*\mlam_2).$$

The fact that the absolute value of a determinant is the product of the singular values, and the inequalities for sums of singular values \cite[Theorem 3.3.16]{HoJ91} implies
\begin{eqnarray*}
\det(\mi_{n-k}+\mg_1\mg_1^*\mlam_2)& \leq& \prod_{j=1}^{n-k} \sigma_j(\mi_{n-k} + \mg_1\mg_1^*\mlam_2)\\
&\leq& \prod_{j=1}^{n-k} (1+ \sigma_j(\mg_1\mg_1^*\mlam_2)).
\end{eqnarray*}
Observe that  $\|\mg_1\|_2\leq \gamma^{q-1/2}\, \|\mom_2\mom_1^{\dagger}\|_2$ and apply the singular value product inequalities~\cite[Theorem 3.3.16(d)]{HoJ91} 

$$\sigma_j(\mg_1\mg_1^*\mlam_2) \leq \sigma_1(\mg_1\mg_1^*)\sigma_j(\mlam_2) \leq {\gamma^{2q-1}}\>\normtwo{\mom_2\mom_1^\dagger}^2\>\lambda_{k+j},
\qquad 1\leq j\leq n-k,$$
and therefore
 $$\det(\mi_{n-k} + \mg_1\mg_1^*\mlam_2) \leq  
\det\left(\mi_{n-k} +\gamma^{2q-1}\>
\normtwo{\mom_2\mom_1^\dagger}^2\> \mlam_2\right).$$
Thus 
$$\mathcal{E}\leq f(\mi_k+\mk) \leq 
f\left(\mi_{n-k} + \gamma^{2q-1}\>\normtwo{\mom_2\mom_1^\dagger}^2\>
\mlam_2\right).$$ 
\end{description}

\paragraph{Derivation of (\ref{e_ld2}) in Steps 4b--5b}
\begin{description}
\item\textit{4b.  Upper bound for $\mathcal{E}$ in (\ref{e_calE}).}
For the matrix $\mh_1$ in (\ref{e_H}) write
\begin{eqnarray*}
\mi_k+\mh_1& = &
(\mi_k+\mf^*\mf)^{-1/2}\,\left((\mi_k+\mf^*\mf)+\mlam_1\right)\,(\mi_k+\mf^*\mf)^{-1/2}\\
&\succeq&\> \mh_4\equiv
(\mi_k+\mf^*\mf)^{-1/2}\,\left(\mi_k+\mlam_1\right)\,(\mi_k+\mf^*\mf)^{-1/2}.
\end{eqnarray*}
Thus $f(\mh_4)=f(\mi_k+\mlam_1)-f(\mi_k+\mf^*\mf)$.
The properties of the Loewner partial order \cite[Corollary 7.7.4]{HoJ13} imply
$$\mathcal{E}\leq f(\mi_k+\mlam_1)-f(\mh_4)\leq f(\mi_k+\mf^*\mf).$$

\item\textit{5b. Bounding $f(\mi_k+\mf^*\mf)$.}
Write 
$$\mi_k+\mf^*\mf = \mi_k+\mg_2^*\mlam_2\mg_2 \qquad \text{where} \quad 
\mg_2 \equiv \mlam_2^{q-1/2}\mom_2\mom_1^{\dagger}\mlam_1^{-q}.$$ Observe that,
$\mg_2 = \mg_1\mlam_1^{-1/2}$ and therefore, $\|\mg_2\|_2\leq \tfrac{\gamma^{q-1/2}}{\sqrt{\lambda_k}}\, \|\mom_2\mom_1^{\dagger}\|_2$. 
The rest of the proof follows the same steps as Step~6a and will not be repeated here.~\qed
\end{description}
\end{proof}

The following discussion mirrors that in Remark~\ref{r_te}.

\begin{remark}\label{r_ld}
The two special cases below illustrate that,
even in a best-case scenario, the accuracy of $\logdet(\mi_{\ell}+\mt)$ is limited by 
$\logdet(\mi_{n-k}+\mlam_2)$.

\begin{itemize}
\item If $\ell=k$ and $\mom=\mU_1$ then
$$\logdet(\mi_n+\ma) - \logdet(\mi_{\ell}+\mt) = \logdet(\mi_{n-k}+\mlam_2).$$
This follows from Lemma~\ref{l_ldel}, and from both bounds in Theorem~\ref{t_ld}
with $\mom_1=\mi_k$ and $\mom_2=\vzero$.
 \item If $\ell>k$ and $\mom$ consists of the columns of $\mU$ associated
with the dominant eigenvalues $\lambda_1, \ldots,\lambda_{\ell}$ of $\ma$, then
\begin{eqnarray*}
\logdet(\mi_{n-k}+\mlam_2) - \log{\prod_{j=k+1}^{\ell}{(1+\lambda_j)}}&\leq& 
\logdet(\mi_n+\ma) - \logdet(\mi_{\ell}+\mt)\\
& \leq&  \logdet(\mi_{n-k}+\mlam_2).
\end{eqnarray*}
This follows from Lemma~\ref{l_ldel}, and from both bounds in Theorem~\ref{t_ld}
with $\mom_1=\begin{pmatrix}\mi_k & \vzero_{k\times (k-\ell)}\end{pmatrix}$ and
$\mom_2=\begin{pmatrix} \vzero_{(n-k)\times k} & *\end{pmatrix}$.

Theorem~\ref{t_ld} cannot be tight for the same reasons as in Remark~\ref{r_te}.
\end{itemize}
\end{remark}
\medskip

The proof for the following bounds is very similar to that of Lemma~\ref{l_ldel}
and Theorem~\ref{t_ld}.

\begin{theorem}\label{t_ld2}
In addition to the assumptions in Section~\ref{s_ass}, let $\ma$ be positive definite; and
let $\mt=\mq^*\ma\mq$ be computed by Algorithm~\ref{alg:randsubspace}. Then
\begin{eqnarray*}
0\leq \lefteqn{\log{(\det\ma)}-\log{\det(\mt)}\leq}\\
&& \qquad\qquad
 \log{\det\left(\mlam_2\right)}+ \log{\det\left(\mi_{n-k}+\frac{\gamma^{2q-1}}{\lambda_k} \,
\normtwo{\mom_2\mom_1^\dagger}^2\>\mlam_2\right)}.
\end{eqnarray*}
\end{theorem}

\section{Probabilistic analysis}\label{s_prob}
We derive probabilistic bounds for $\normtwo{\mom_2\mom_1^{\dagger}}$, a term that
represents the contribution of the starting guess $\mom$ in the dominant 
eigenspace $\mU_1$ of $\ma$, when the elements of $\mom$ are either Gaussian random variables
(Section~\ref{s_probg}) or Rademacher random variables (Section~\ref{s_probr}).

The theory for  Gaussian random matrices suggests the value $p \lesssim 20$ whereas theory for Rademacher random matrices suggests that $\ell \sim (k+\log n)\log k$ samples need to be taken to ensure 
$\rank(\mom_1) = k$. However, the theory for Rademacher random matrices is pessimistic, and numerical experiments demonstrate that a practical value of $p\lesssim 20$ is sufficient. 

\subsection{Gaussian random matrices}\label{s_probg}
For the Gaussian starting guess, we present bounds for expectation

We split our analysis into two parts: an average case analysis (Section~\ref{ss_g_exp}) and a 
concentration inequality (Section~\ref{ss_g_exp}), and prove Theorem~\ref{t_gauss_exp}.

\begin{definition}
A ``standard'' Gaussian matrix has elements that are independently and identically distributed
 random $\mathcal{N}(0,1)$ variables, that is normal random variables with mean~0 and variance~1.
 \end{definition}
 
 Appendix~\ref{app_gauss} summarizes the required results for standard Gaussian matrices.
 In particular, we will need the following.
 
\begin{remark}\label{r_g}
The distribution of a standard Gaussian matrix $\mg$ is rotationally invariant.
That is, if $\mU$ and $\mv$ are unitary matrices, 
then $\mU^*\mg\mv$ has the same distribution as $\mg$. Due to this property, the contribution of the starting guess on the dominant eigenspace does not appear in the bounds below.
\end{remark}

\subsubsection{Average case analysis} \label{ss_g_exp}
We present bounds for the expected values of $\normtwo{\mg}^2$ and $\normtwo{\mg^\dagger}^2$ 
for standard Gaussian matrices $\mg$, and then prove Theorem~\ref{t_gauss_exp}.

\begin{lemma}\label{p_2norm2}
Draw two Gaussian random matrices  $\mg_2\in\mathbb{R}^{(n-k) \times (k+p)} $ and $\mg_1\in\mathbb{R}^{k \times (k+p)}$ and let $p\geq 2$. With $\mu$ defined in~\eqref{e_mu}, then 
\begin{equation}\label{2norm2}
\expect{ \normtwo{\mg_2}^2 }  \leq \>{\mu}^2 + 2\left(\sqrt{\frac{\pi}{2}}\,{\mu} +1 \right).
\end{equation}
If, in addition, also $k \geq 2$, then
\begin{equation}\label{pinv2norm2}
\expect{\normtwo{\mg_1^\dagger}^2 }  \leq \> \frac{p+1}{p-1}\left(\frac{1}{2\pi\,(p+1)}\right)^{2/(p+1)}\left(\frac{e\,\sqrt{k+p}}{p+1}\right)^2. 
\end{equation}
\end{lemma}

\begin{proof} See Appendix~\ref{app_gauss}.
\end{proof}

We are ready to derive the main theorem on the expectation of standard Gaussian matrices. 
 
\paragraph{Proof of Theorem~\ref{t_gauss_exp}}
We start as in the proof of \cite[Theorem~10.5]{HMT09}.
The assumptions in Section~\ref{s_ass} and Remark~\ref{r_g} imply that
$\mU^*\mom$ is a standard Gaussian matrix. Since
 $\mom_1$ and $\mom_2$ are non-overlapping submatrices of $\mU^*\mom$, 
they are both standard Gaussian and stochastically independent.
 The sub-multiplicative property implies 
 $\normtwo{\mom_2\mom_1^\dagger} \leq \normtwo{\mom_2}\normtwo{\mom_1^\dagger}$. 

We use independence of $\mom_2$ and $\mom_1$
and apply both parts of (\ref{pinv2norm2});
with $\mu\equiv\sqrt{n-k}+\sqrt{k+p}$ this gives
\begin{eqnarray}\label{2norm_exp}
\expect{\normtwo{\mom_2\mom_1^\dagger}^2} \leq \> \left[\mu^2+ 2\left(\sqrt{\frac{\pi}{2}}\mu +1 \right)\right]\frac{p+1}{p-1}\left(\frac{1}{2\pi(p+1)}\right)^{\frac{2}{(p+1)}}\left(\frac{e\sqrt{k+p}}{p+1}\right)^2.
\end{eqnarray}
Bounding
$$\mu^2 + 2\left(\sqrt{\frac{\pi}{2}}\mu +1 \right) \leq (\mu + \sqrt{2})^2.$$
gives $\expect{\normtwo{\mom_2\mom_1^\dagger}^2} \leq {C}_{ge}$. 
Substituting into the result of Theorem~\ref{t_te} we have the desired result in Theorem~\ref{t_gauss_exp}.

For a positive definite matrix $\logdet(\ma) = \trace(\log(\ma))$, therefore 
\[ \logdet(\mi+\gamma^{2q-1}\normtwo{\mom_2\mom_1^\dagger}^2\mlam_2) = \sum_{j=k+1}^{n} \log\left(1 + \gamma^{2q-1}\normtwo{\mom_2\mom_1^\dagger}^2 \lambda_{j}\right). \]
Observe that $ \log(1+\alpha x)$ is a concave function. Using Jensen's inequality, for $j=k+1,\dots,n$
\[\expect{\log(1 + \gamma^{2q-1}\lambda_j\normtwo{\mom_2\mom_1^\dagger}^2)} \leq \log\left(1 +  \gamma^{2q-1}\lambda_j\expect{\normtwo{\mom_2\mom_1^\dagger}^2}\right). \]

Since $\log(1+\alpha x)$ is monotonically increasing function, the second result in Theorem~\ref{t_gauss_exp} follows by substituting the above equation into the bounds from Theorem~\ref{t_ld} and simplifying the resulting expressions.
\cvd

\subsubsection{Concentration inequality} \label{ss_g_dev}

As with the expectation bounds, it is clear that we must focus our attention on the term $\normtwo{\mom_2} \normtwo{\mom_1^\dagger}$. We reproduce here a result on on the concentration bound of this term. The proof is provided in~\cite[Theorem 5.8]{gu2015subspace}. 
\begin{lemma}\label{p_gauss_dev}
Let $\mom_2 \in \mathbb{R}^{(n-k)\times(k+p)}$ and $\mom_1 \in \mathbb{R}^{k\times (k+p)}$ be two independent Gaussian random matrices and let $p \geq 4$. Then for $ 0 < \delta < 1$, 
\begin{equation}\label{O2O1_dev}
\prob{ \normtwo{\mom_2}^2 \normtwo{\mom_1^\dagger}^2 \geq {C}_g} \leq \delta,
\end{equation}•
where ${C}_g$ is defined in Theorem~\ref{t_gauss}. 
\end{lemma}

The following statements mirror the discussion in~\cite[Section 5.3]{gu2015subspace}. While the oversampling parameter $p$ does not significantly affect the expectation bounds as long as $p \geq 2$, it seems to affect the concentration bounds significantly. The oversampling parameter $p$ can be chosen in order to make $(2/\delta)^{1/(p+1)}$ a modest number, say less than equal to $10$. Choosing 
\[ p = \left\lceil\log_{10}\left(\frac{2}{\delta}\right) \right\rceil -1, \] 
for $\delta = 10^{-16}$ gives us $p = 16$. In our experiments, we choose the value for the oversampling parameter $p = 20$.

\subsection{Rademacher random matrices}\label{s_probr}
We present results for the concentration bounds when $\mom$ is a Rademacher random matrix. We start with the following definition.

\begin{definition}
A  Rademacher random matrix has elements that are independently and identically distributed and take on values $\pm 1$ with equal probability. \end{definition}

Note that unlike standard Gaussian matrices, the distribution of a Rademacher matrix  is not rotationally invariant.

As before we partition $\mU = \begin{bmatrix} \mU_1 &\mU_2\end{bmatrix}$ and let $\mom_1 = \mU_1^*\mom$ and $\mom_2 = \mU_2^*\mom$. The following result bounds the tail of $\normtwo{\mom_2}^2\normtwo{\mom_1^\dagger}^2$. This result can be used to readily prove Theorem~\ref{t_rademacher}.
\begin{theorem}\label{t_rad}
Let $ \rho \in (0,1) $ and $0 < \delta < 1$ and integers $n,k \geq 1$. Let the number of samples $\ell$ be defined in Theorem~\ref{t_rademacher}. Draw a random Rademacher matrix $\mom \in \mathbb{R}^{n\times \ell}$. Then  
\[ \prob{ \normtwo{\mom_2}^2\normtwo{\mom_1^\dagger}^2 \geq {C}_r } \leq \delta,  \]
where $\mathcal{C}_r$ is defined in Theorem~\ref{t_rademacher}.
\end{theorem}
\begin{proof}
See Appendix~\ref{app_rad}.~\qed
\end{proof}

\begin{remark}
From the proof of Theorem~\ref{t_rad}, to achieve $\normtwo{\mom_1^\dagger}  \geq 3 /\sqrt{\ell}$ with at 
least $99.5\%$ probability and $n=1024$, the number of samples required are 
\[\ell \geq 2.54 (\sqrt{k} + 11 )^2(\log(4k) + 4.7).  \]
Here $\rho = 8/9$ is chosen to be so that $1/(1-\rho) = 9$. 
\end{remark}

\begin{figure}\centering
\includegraphics[scale=0.5]{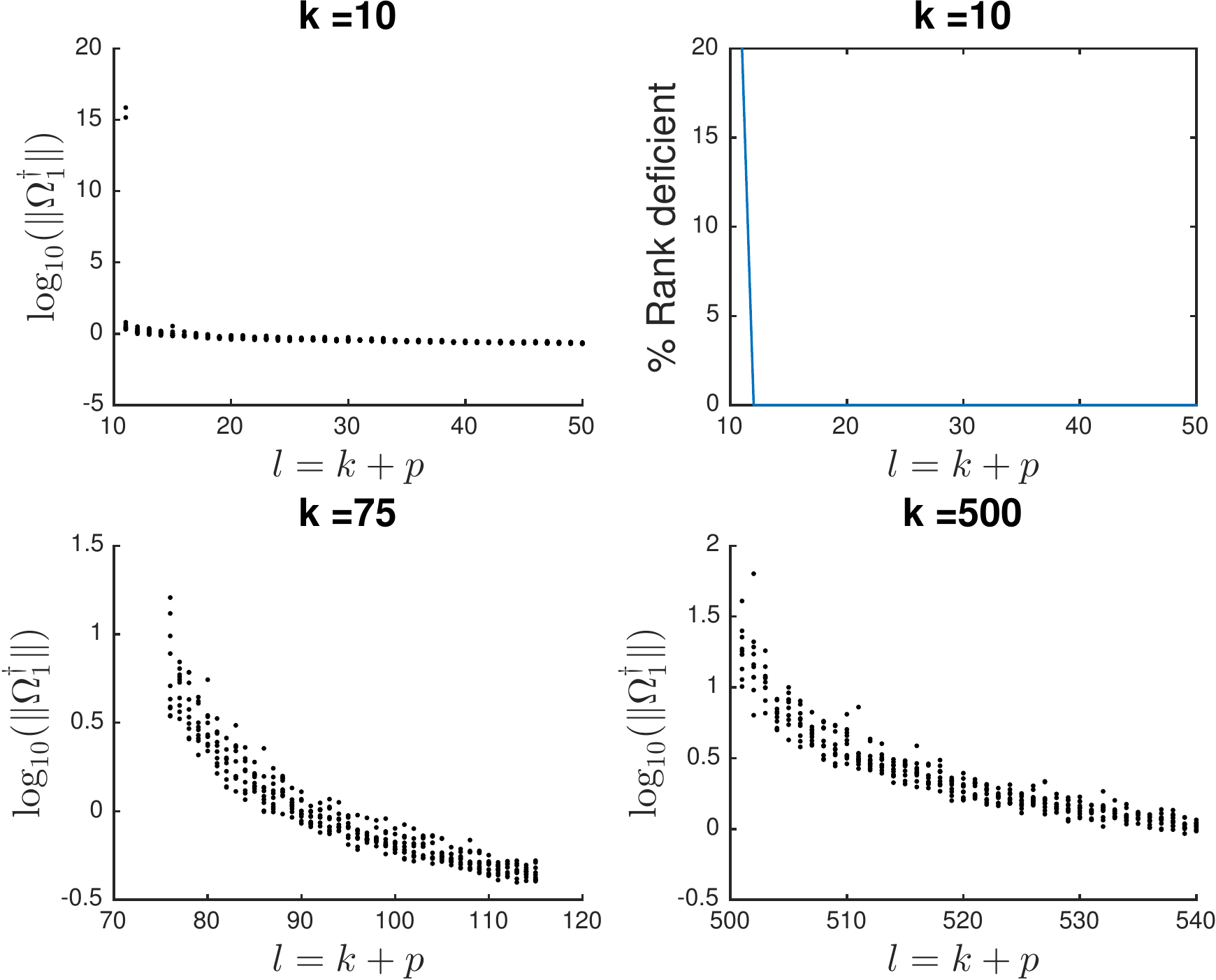}
\caption{The quantity $\log_{10}(\normtwo{\mom_1^\dagger})$ for different amounts of sampling $\ell = k+1$ to $\ell = k+40$. We consider $k=10$ (top left), $k=75$ (bottom left), and $k=500$ (bottom right). The top right plot shows the percentage of numerically rank deficient sampled matrices. }
\label{fig:radpinv}
\end{figure}
The imposition that $\ell \leq n$ implies that the bound may only be informative for $k$ small enough. Theorem~\ref{t_rad} suggests that the number of samples $\ell \sim (k+\log n) \log k$ to ensure that $\normtwo{\mom_1^\dagger}$ is small and $\rank(\mom_1) = k$.

We investigate this issue numerically. We generate random Rademacher matrices $\mom_1 \in \mathbb{R}^{k\times \ell}$; here we assume $\mU = \mi_n$. Here we choose three different values for $k$, namely $k = 10,75,500$. For each value of $k$, the oversampling $\ell$ varies from $\ell = k+1$ to $\ell = k+40$. We generate $500$ runs for each value of $\ell$. In Figure~\ref{fig:radpinv} we plot $\normtwo{\mom_1^\dagger}^2$ and the percentage that are rank deficient. For $k=10$, a few samples were rank deficient but the percentage of rank deficient matrices dropped significantly; after $p=20$ there were no rank deficient matrices. For larger values of $k$ we observed that none of the sampled matrices were rank deficient and $p=20$ was sufficient to ensure that $\normtwo{\mom_1^\dagger}^2 \lesssim 10$. Similar results were observed for randomly generated orthogonal matrices $\mU$. In summary, a modest amount of oversampling $p \lesssim 20$ is sufficient to ensure that $\rank(\mom_1) = k$ for the Rademacher random matrices, similar to Gaussian random matrices. In further, numerical experiments we shall use this particular choice of oversampling parameter $p$.

\section{Numerical Experiments}\label{num_small}
In this section, we demonstrate the performance of our algorithm and bounds on two different kinds of examples. In the first example, we focus on small matrices (with dimension $128$) in the non-asymptotic regime. We show that our bounds are informative and our estimators are accurate with high probability. In the second examples, we look at medium sized matrices (of dimension $5000$) and demonstrate the behavior of our estimators. 
\subsection{Small matrices}
In this section we study the performance of the proposed algorithms on small test examples. 

The  $\ma$ is chosen to be of size $128\times 128$ and its eigenvalues satisfy $\lambda_{j+1} = \gamma^j \lambda_1$ for $j=1,\dots,n-1$. %
To help interpret the results of Theorems~\ref{t_gauss_exp}-\ref{t_rademacher}, we provide simplified versions of the bounds. The relative error in the trace estimator can be bounded as 
\begin{equation}\label{e_cond_err}
 \Delta_t \equiv \frac{\trace(\ma) - \trace(\mt)}{\trace(\ma)} \leq \>  (1+\gamma^{2q-1}C)\frac{\gamma^k(1-\gamma^{n-k})}{1-\gamma^n}. 
\end{equation}
Here $C$ can take the value $C_g$ for Gaussian starting guess and $C_r$ for Rademacher starting guess.

For the logdet estimator, we observe that $\log(1+x) \leq x$. Using the
relation $\logdet(\mi+\mlam) = \trace\log(\mi+\ma)$, it is reasonable to bound
$\logdet(\mi+\mlam)$ by $\trace(\mlam)$. With the abbreviation $f(\cdot) =
\log\det(\cdot)$  we can bound the relative error of the logdet estimator as 
\[ \Delta_l \equiv \frac{f(\mi+\ma) - f(\mi+\mt)}{f(\mi+\ma)} \leq (1+\gamma^{2q-1}C)\frac{\gamma^k(1-\gamma^{n-k})}{1-\gamma^n}. \]

Consequently, the error in the trace and logdet approximations approach $0$ as $k\rightarrow n$ and is equal to $0$ if $k=n$.

In the following examples, we study the performance of the algorithms with increasing sample size. It should be noted here that, since $\ell = k+p$ and $p$ is fixed, increasing sample size corresponds to increasing the dimension $k$; consequently, the location of the gap is changing, as is the residual error $\Delta = \trace(\mlam_2)/\trace(\mlam)$.

\begin{figure}[!ht]\centering
\includegraphics[scale=0.3]{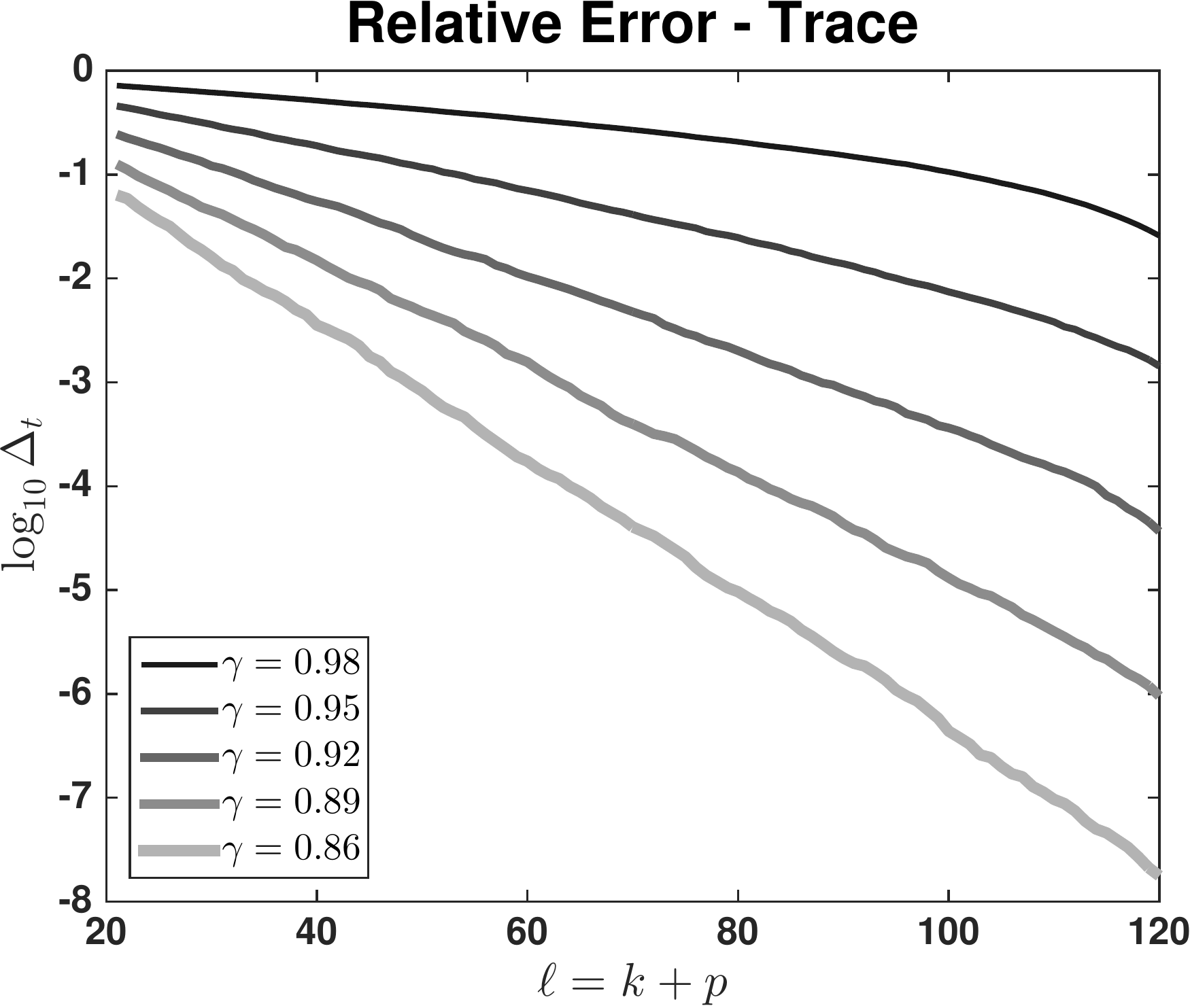}
\includegraphics[scale=0.3]{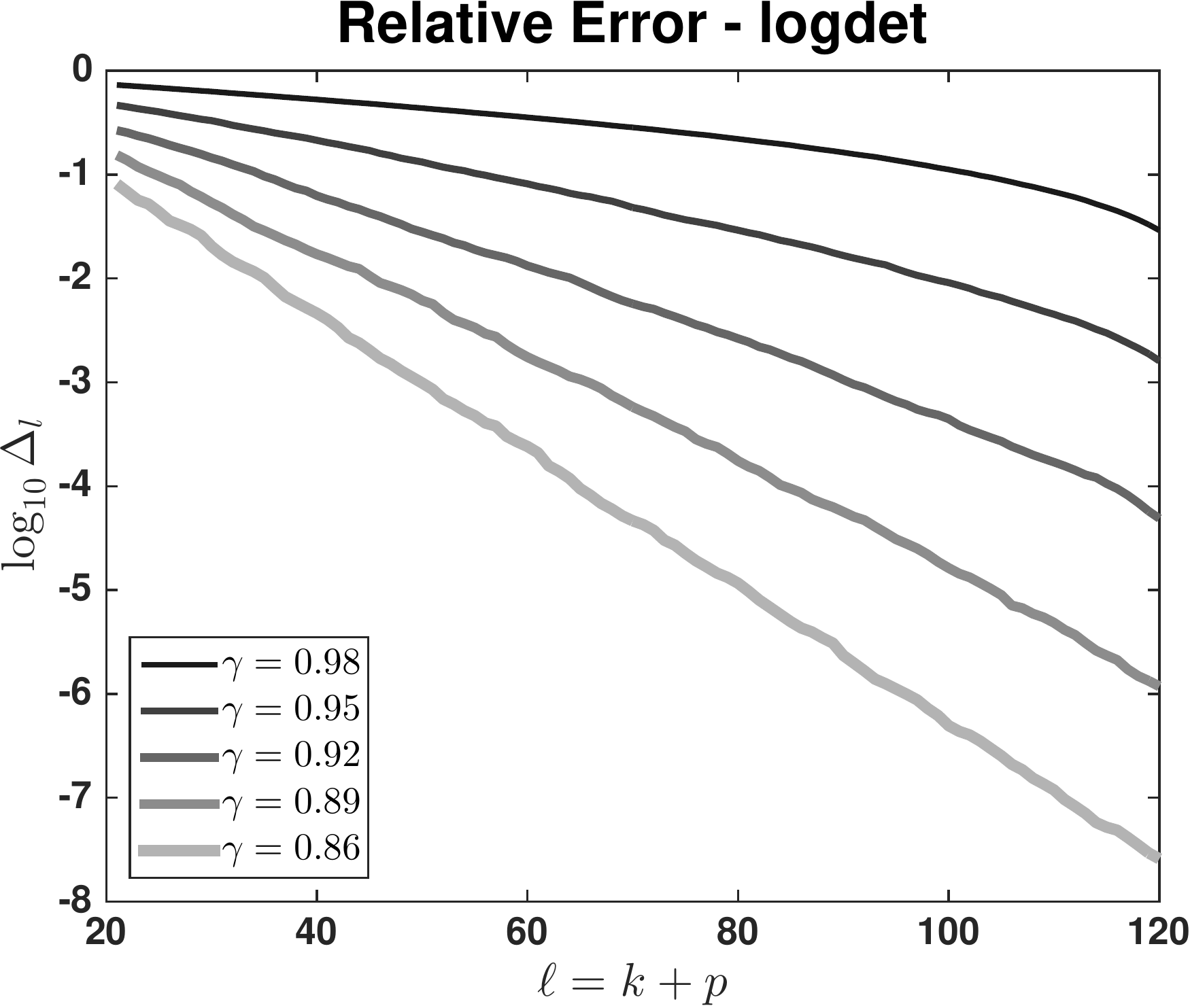}
\caption{Accuracy of proposed estimators on matrix with geometrically decaying  eigenvalues.  The relative error is plotted against the sample size. Accuracy of (left) trace and (right) logdet estimators. Here Gaussian starting guess was used.}
\label{f_cond}
\end{figure}

\paragraph{1. Effect of eigenvalue gap}
Matrices $\ma$ are generated with different eigenvalue distributions using the example code above. The eigenvalue gap parameter $\gamma$ varies from $0.98$ to $0.86$. We consider sampling from both Gaussian random matrices. The oversampling was set to be $p=20$ for both distributions. The subspace iteration parameter $q$ was set to be $1$. The results are displayed in Figure~\ref{f_cond}. Clearly, both the trace and logdet become increasingly accurate as the eigenvalue gap increases. This confirms the theoretical estimate in~\eqref{e_cond_err} since the error goes to zero as $k\rightarrow n$. The behavior of the error with both Gaussian and Rademacher starting guesses is very similar and is not displayed here.

\paragraph{2. Comparison with Monte Carlo estimators}
We fix the eigenvalue gap to $\gamma = 0.9$, sampling parameter $p=20$ and subspace iteration parameter $q=1$.  We consider sampling from both Gaussian and Rademacher random matrices and consider their accuracy against their Monte Carlo counterparts. As mentioned earlier, the Monte Carlo estimator cannot be directly applied to the logdet estimator; however, using the following identity
\[ \logdet(\mi_n +\ma) = \trace\log(\mi_n+\ma), \] 
the Monte Carlo estimators can be applied to the matrix $\log(\mi_n+\ma)$. For a fair comparison with the estimators proposed in this paper, the number of samples used equals the target rank plus the oversampling parameter $p=20$, i.e., $(k+p)$ samples. We averaged the Monte Carlo samplers over $100$ independent runs. The results are illustrated in Figure~\ref{f_mc}. It can be readily seen that when the matrix has rapidly decaying eigenvalues, our estimators are much more accurate than the Monte Carlo estimators. The number of samples required for the Monte Carlo methods for a  relative accuracy $\epsilon$ depends as $\epsilon^{-2}$, so the number of samples required for an accurate computation can be large. For the logdet estimator, initially the Monte Carlo estimator seems to outperform our method for small sample sizes; however, the error in our estimators decays sharply. It should be noted that for this small problem one can compute $\log(\mi+\ma)$ but for a larger problem it maybe costly even prohibitively expensive. For all the cases described here, Gaussian and Rademacher random matrices seem to have very similar behavior.  
\begin{figure}[!ht]\centering
\includegraphics[scale=0.3]{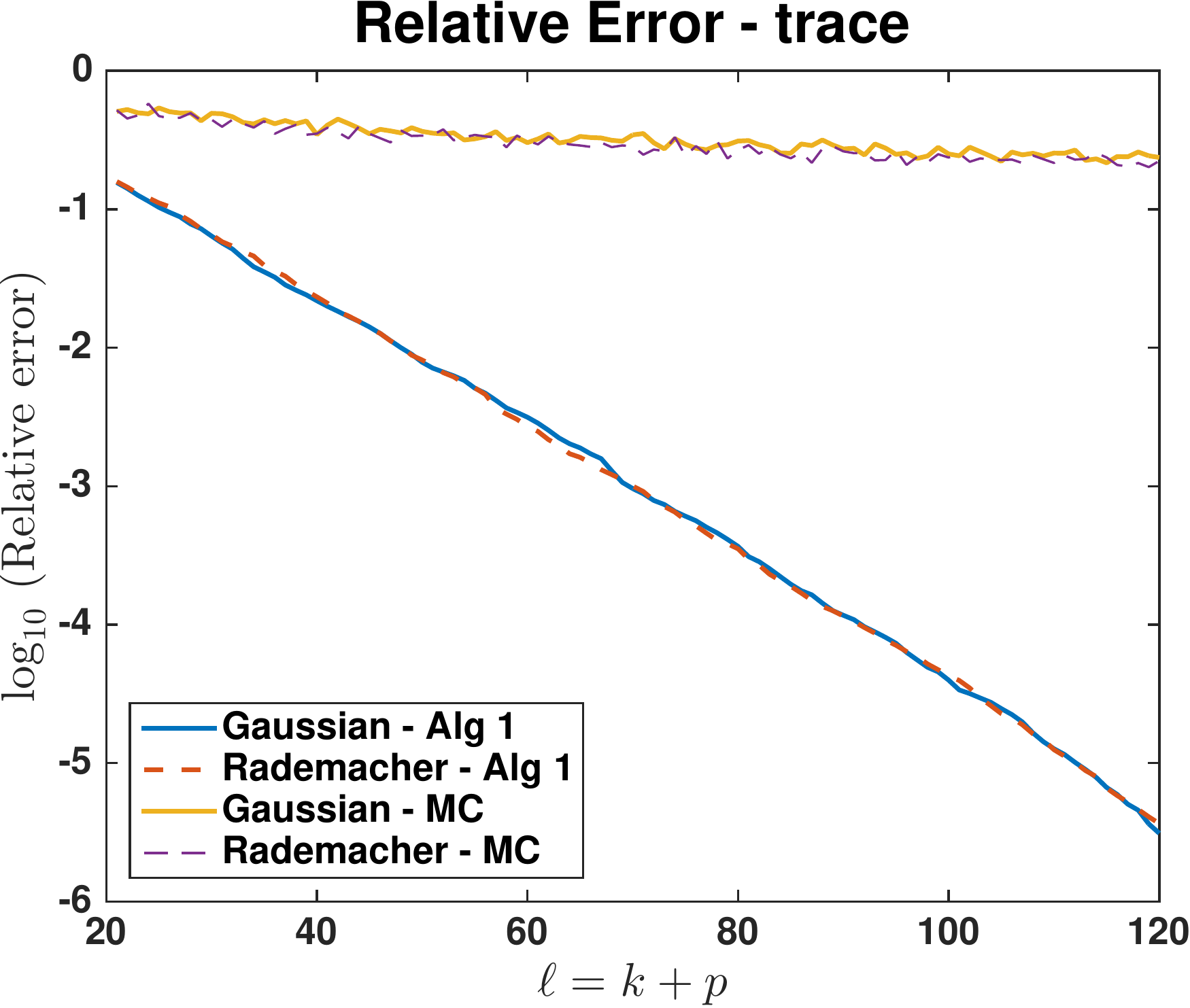}
\includegraphics[scale=0.3]{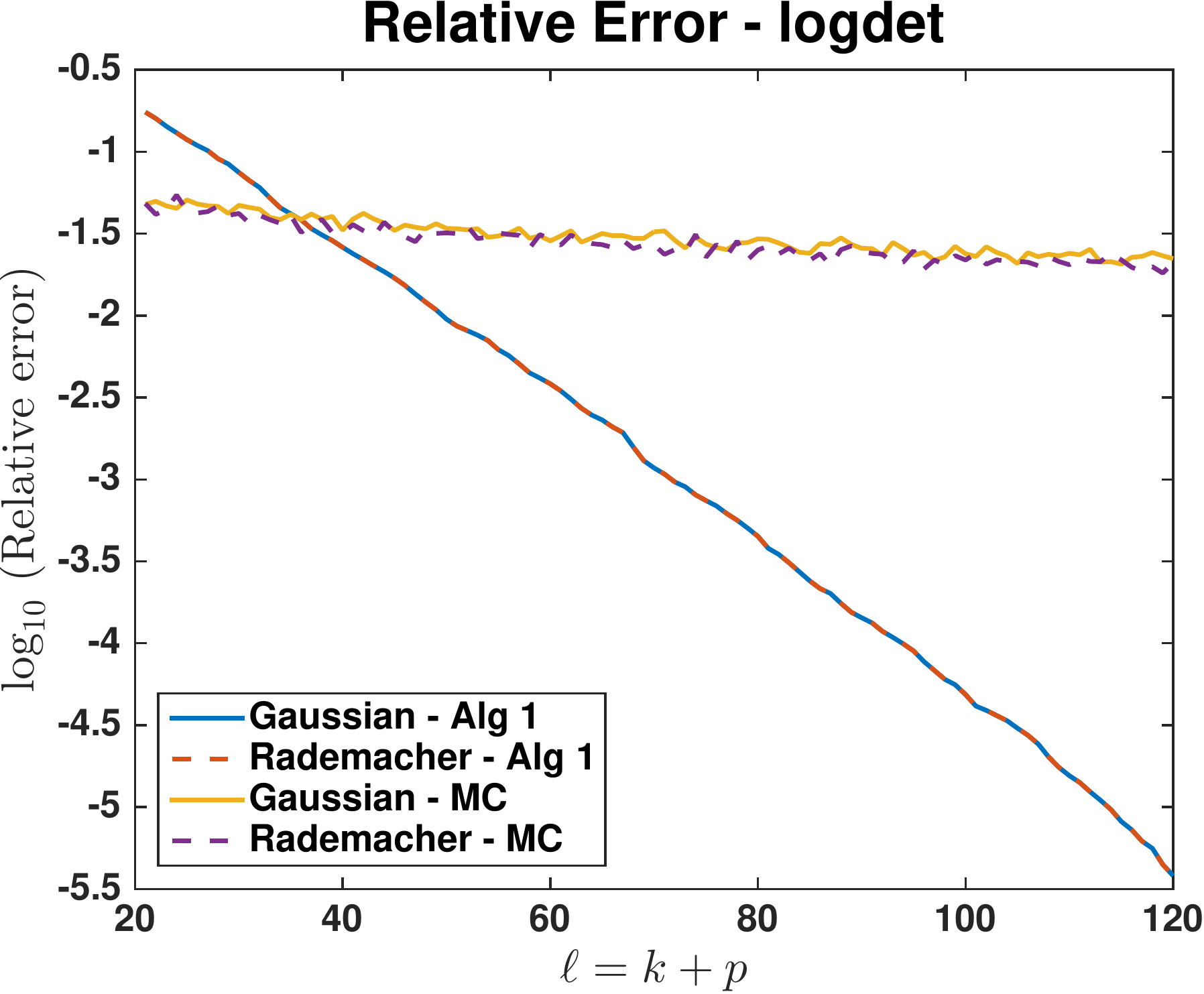}
\caption{Comparison against Monte Carlo estimators for (left) trace and (right) logdet computations.  The relative error is plotted against the sample size. }
\label{f_mc}
\end{figure}

\paragraph{3. Effect of subspace iteration parameter}
The matrix $\ma$ is the same as in the previous experiment but $p$ is chosen to be $0$. The subspace iteration parameter is varied from $q=1$ to $q=5$.  The results of the relative error as a function of $\ell$ are displayed in Figure~\ref{f_tq}. The behavior is similar for both Gaussian and Rademacher starting guess, therefore we only display results for Gaussian starting guess. 
We would like to emphasize that Algorithm~\ref{alg:randsubspace} is not implemented \textit{as is} since it is numerically unstable and susceptible to round-off error pollution; instead a numerically stable version is implemented based on~\cite[Algorithm A.1]{gu2015subspace}. As can be seen, increasing the parameter improves the accuracy for a fixed target rank $k$. However, both from the analysis and the numerical results, this is clearly a case of diminishing returns. This is because the overall error is dominated by $\trace(\mlam_2)$ and $\logdet(\mi_{n-k}+\mlam_2)$. Increasing the subspace iteration parameter $q$, only improves the multiplicative factor in front of one of the terms. Moreover, in the case that the eigenvalues are decaying rapidly, one iteration, i.e. $q=1$ is adequate to get an accurate estimator.  
\begin{figure}[!ht]\centering
\includegraphics[scale=0.3]{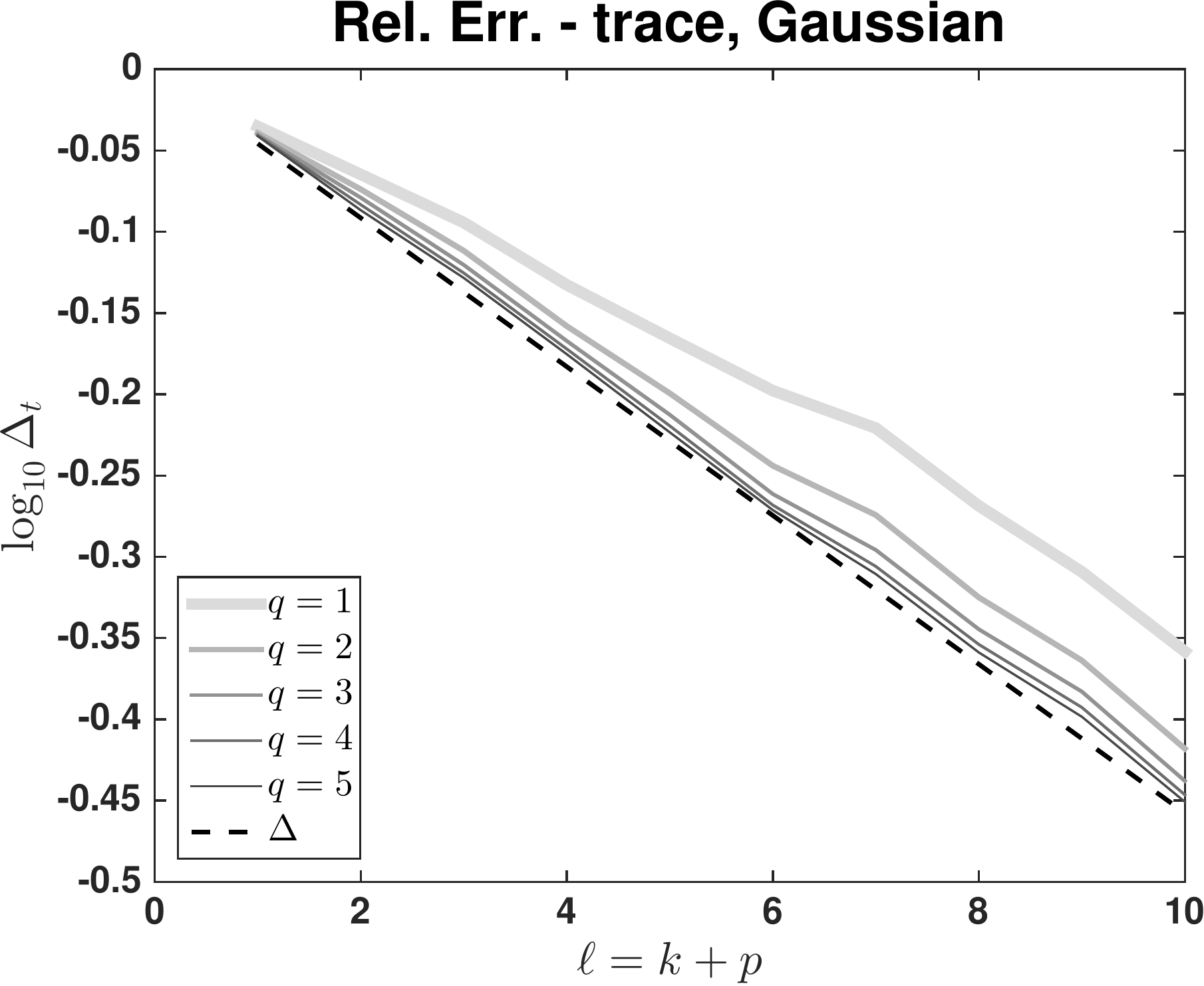}
\includegraphics[scale=0.3]{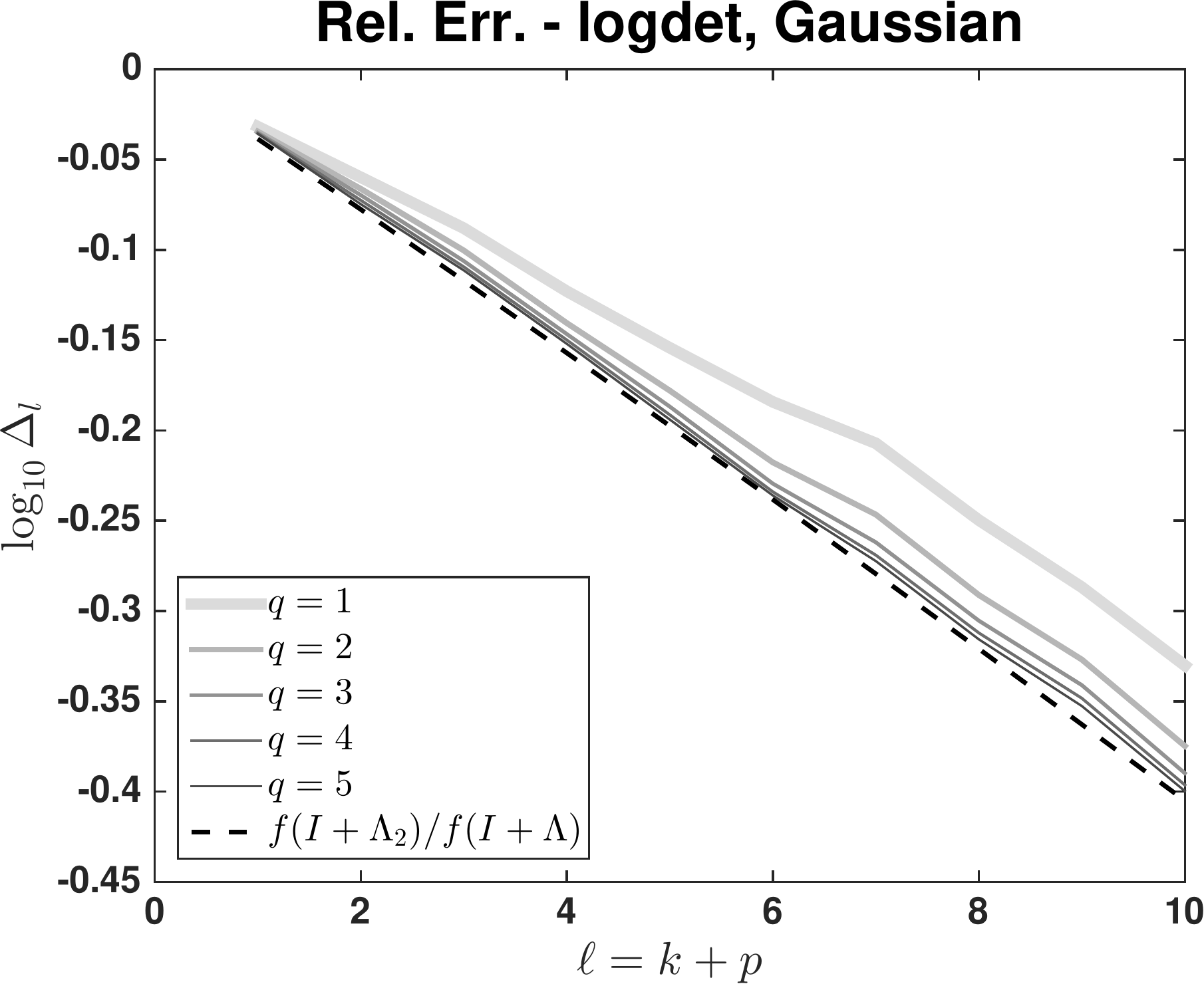}
\caption{Effect of subspace iteration parameter $q$ on the  (left) trace estimator and (right) logdet estimator.  The relative error is plotted against the sample size. Gaussian starting guess was used. The behavior is similar for Rademacher starting guess.}
\label{f_tq}
\end{figure}

\paragraph{4. How descriptive are the bounds?} In this experiment we demonstrate the accuracy of the bounds derived in Section~\ref{s_struct}.  The matrix is chosen to be the same as the one in Experiment 2.  In Figure~\ref{f_err_est} we consider the bounds in the trace estimator derived in Theorem~\ref{t_te}. We consider both the Gaussian (left panel) and Rademacher distributions (right panel). For comparison we also plot the term $\Delta$, which is the theoretical optimum. `Est 1' refers to the first bound in ~\eqref{e_te1} and `Est 2' refers to the second bound in~\eqref{e_te2}. Both the bounds are qualitatively similar to both the true error and the theoretical estimate $\Delta$, and also quantitatively within a factor of $10$ of the theoretical estimate $\Delta$. Since $\gamma$ is close to $1$ and $\normtwo{\mom_2\mom_1^\dagger} > 1$, $\gamma\normtwo{\mom_2\mom_1^\dagger} > 1$ and therefore `Est 1' is a more accurate estimator. The error of the logdet estimator is plotted against the theoretical bounds (see Theorem~\ref{t_ld}) in Figure~\ref{f_err_est_logdet}; as before, our estimator is both qualitatively and quantitatively accurate. The conclusions are identical for both Gaussian and Rademacher matrices. The empirical performance of this behavior is studied in the next experiment.  
\begin{figure}[!ht]\centering
\includegraphics[scale=0.3]{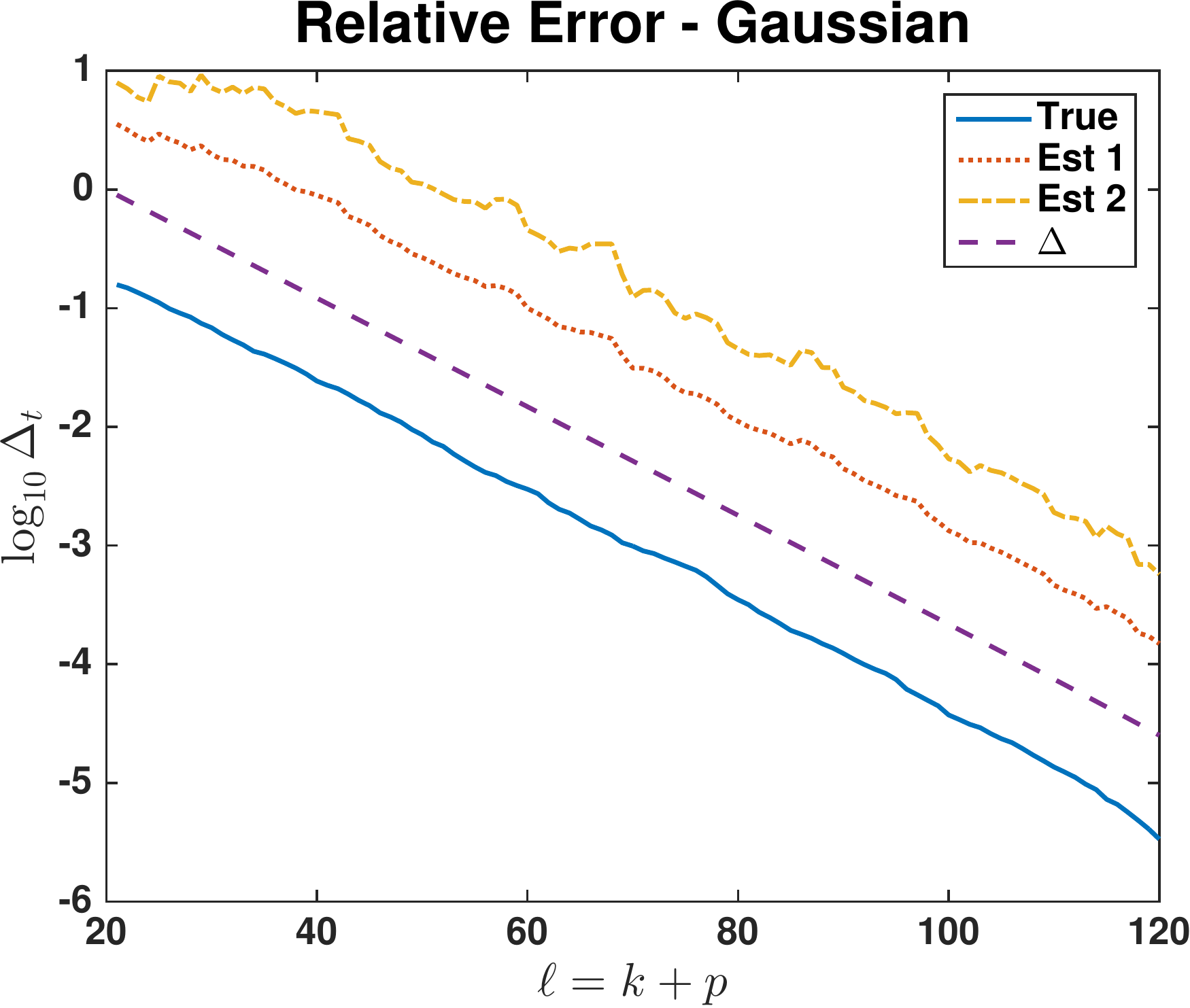}
\includegraphics[scale=0.3]{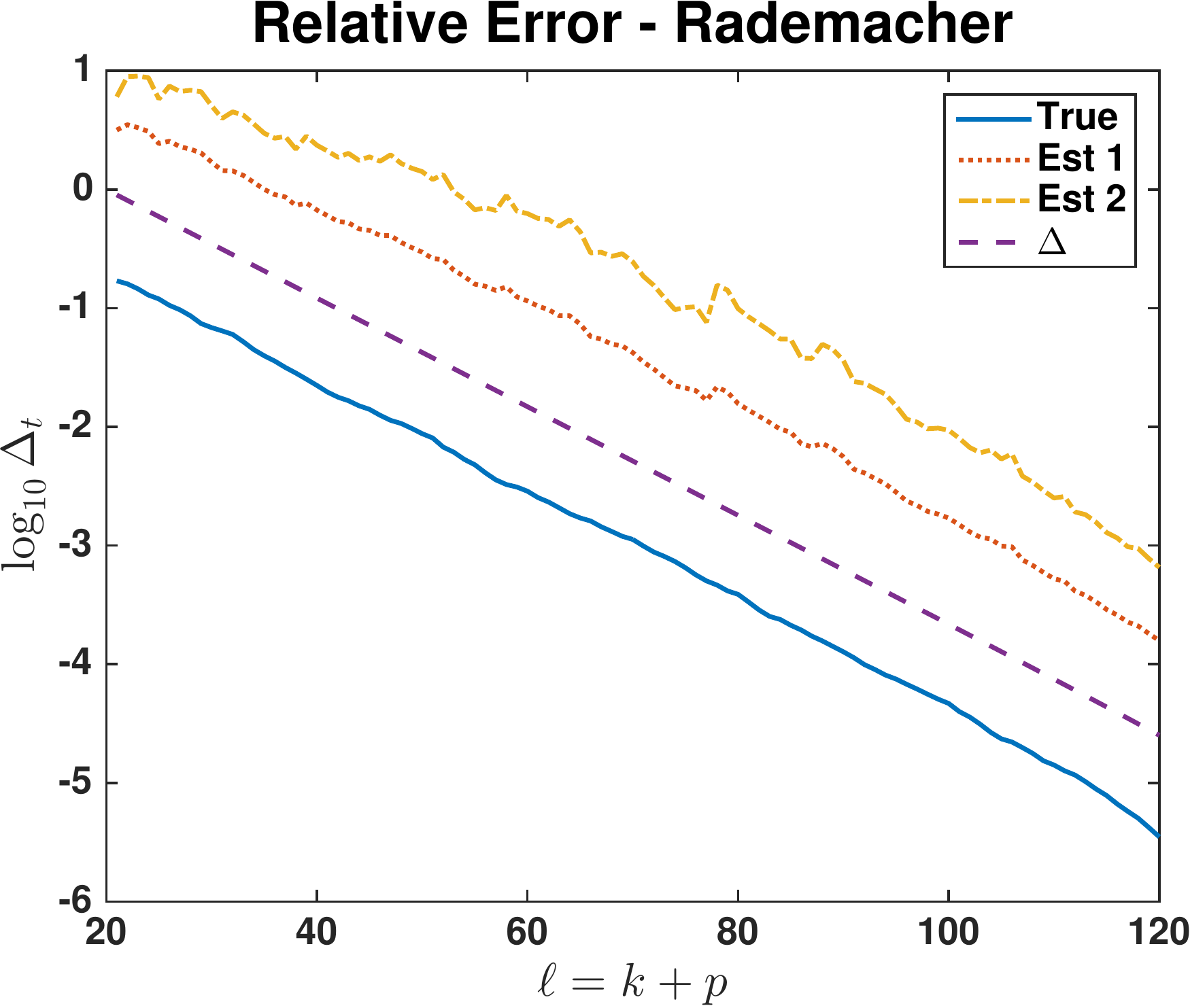}
\caption{Accuracy of the error bounds for the Trace estimators. The relative error is plotted against the sample size. (left) Gaussian and (right) Rademacher random matrices. `Est1'  and `Est2' refers to the bounds in Theorem~\ref{t_te}. For comparison, we also plot $\trace(\mlam_2)/\trace(\ma)$.}
\label{f_err_est}
\end{figure}

\begin{figure}[!ht]\centering
\includegraphics[width=0.5\textwidth]{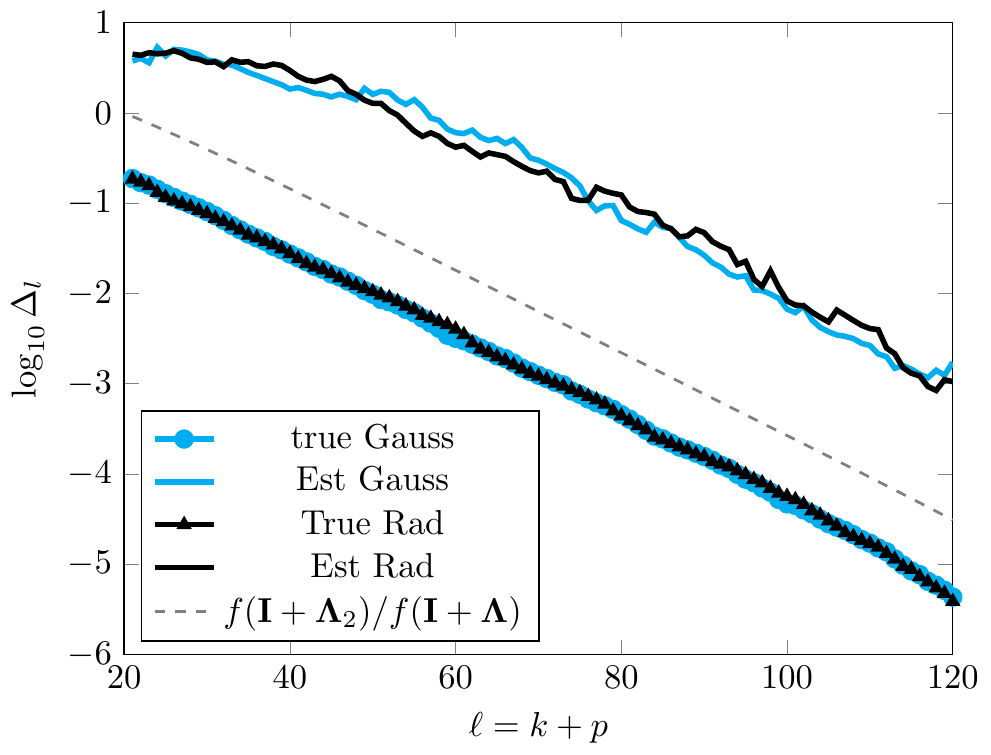}
\caption{Accuracy of the error bounds for the logdet estimators.  The relative error is plotted against the sample size. Both Gaussian and Rademacher starting guesses are used. For comparison, we also plot $\logdet(\mi+\mlam_2)/\logdet(\mi+\ma)$.}
\label{f_err_est_logdet}
\end{figure}

\begin{figure}[!ht]\centering
\includegraphics[scale=0.8]{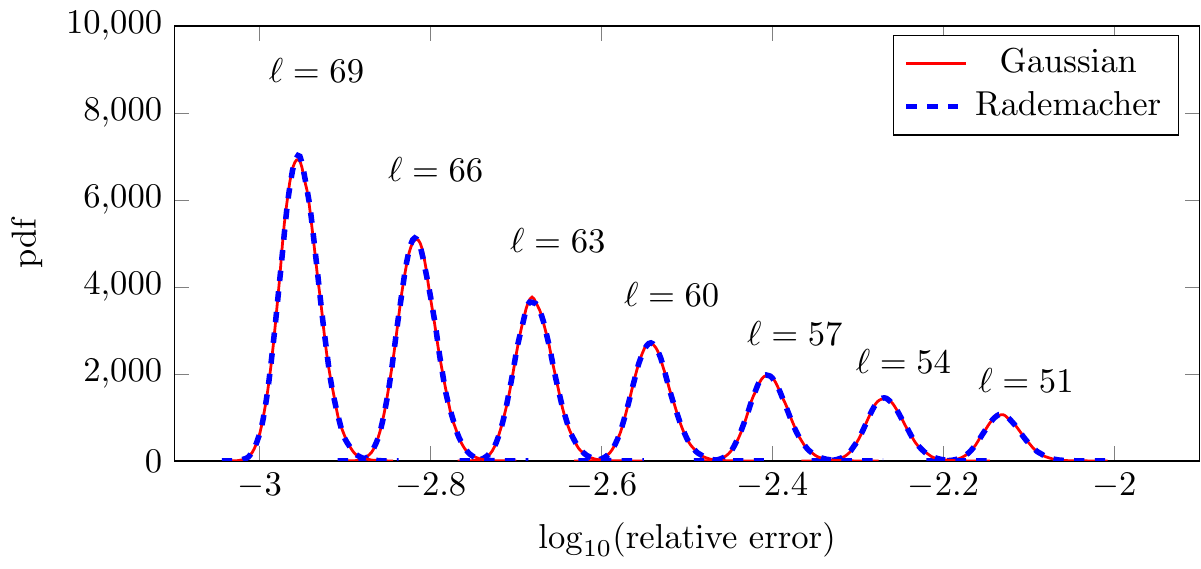}
\caption{Empirical distribution of the relative error for the trace estimator. 
$10^5$ samples were used for the distribution.}
\label{f_emp_trace}
\end{figure}

\paragraph{5. Concentration of measure}
 We choose the same matrix as in Experiment 2. We generate $10^5$ starting guesses (both Gaussian and Rademacher) and compute the distribution of relative errors for the trace (quantified by $\Delta_t$) and logdet (quantified by $\Delta_l$). Figures~\ref{f_emp_trace} and~\ref{f_emp_logdet} show the empirical probability density function for the relative errors in the trace and logdet respectively. We observe that the two distributions are nearly identical and that the empirical density is concentrated about the mean. Furthermore, as the sample size $\ell$ increases, the both the mean and variance of the empirical distribution decrease. These results demonstrate that the randomized methods are indeed effective with high probability.

\begin{figure}[!ht]\centering
\includegraphics[scale=0.8]{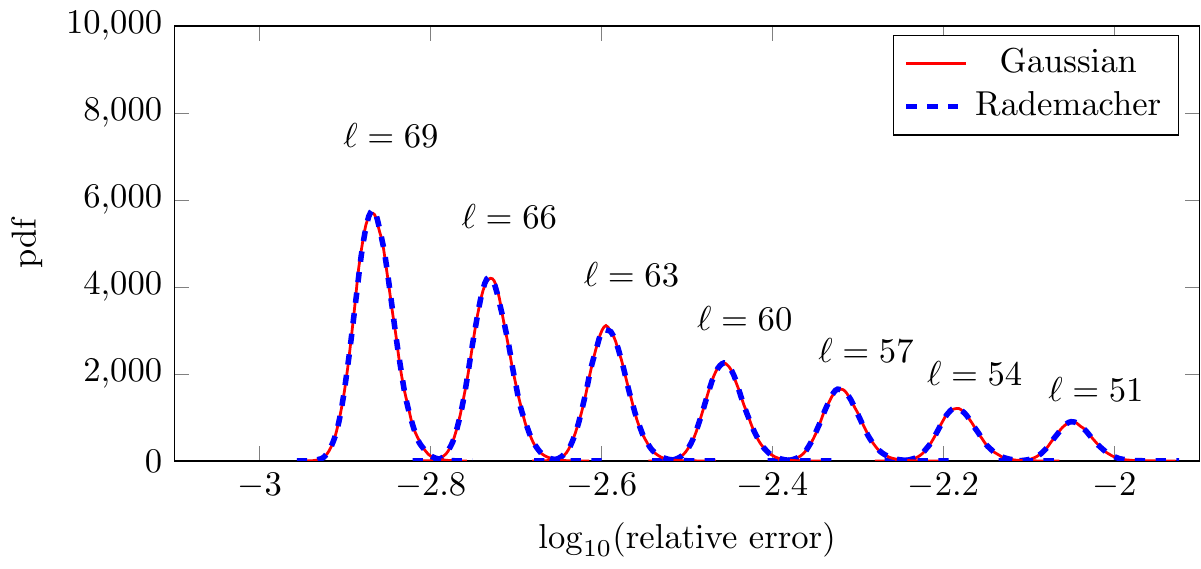}
\caption{Empirical distribution of the relative error for the logdet estimator. 
$10^5$ samples were used for the distribution.}
\label{f_emp_logdet}
\end{figure}

\subsection{Medium sized example}

This example is inspired by a test case from Sorensen and Embree~\cite{SoE14}. Consider the matrix $\ma\in \mathbb{R}^{5000\times 5000}$ defined as 
\begin{equation}\label{e_medium}
\ma \equiv \sum_{j=1}^{40} \frac{h}{j^2}\vx_j\vx^T_j +\sum_{j=41}^{300} \frac{l}{j^2}\vx_j\vx^T_j,
\end{equation}
where $\vx_j \in \mathbb{R}^{5000}$ are sparse vectors with random nonnegative entries. In MATLAB this can be generated using the command
\verb|xj = sprand(5000,1,0.025)|. It should be noted the vectors $\vx_j$ are not orthonormal; therefore, the outer product form is not the eigenvalue decomposition of the matrix $\ma$. However, the eigenvalues decay like $1/j^2$ with a gap at index $40$, and its magnitude  depends on the ratio $h/l$. The exact rank of this matrix is $300$.

\begin{figure}[!ht]\centering
\includegraphics[scale=0.3]{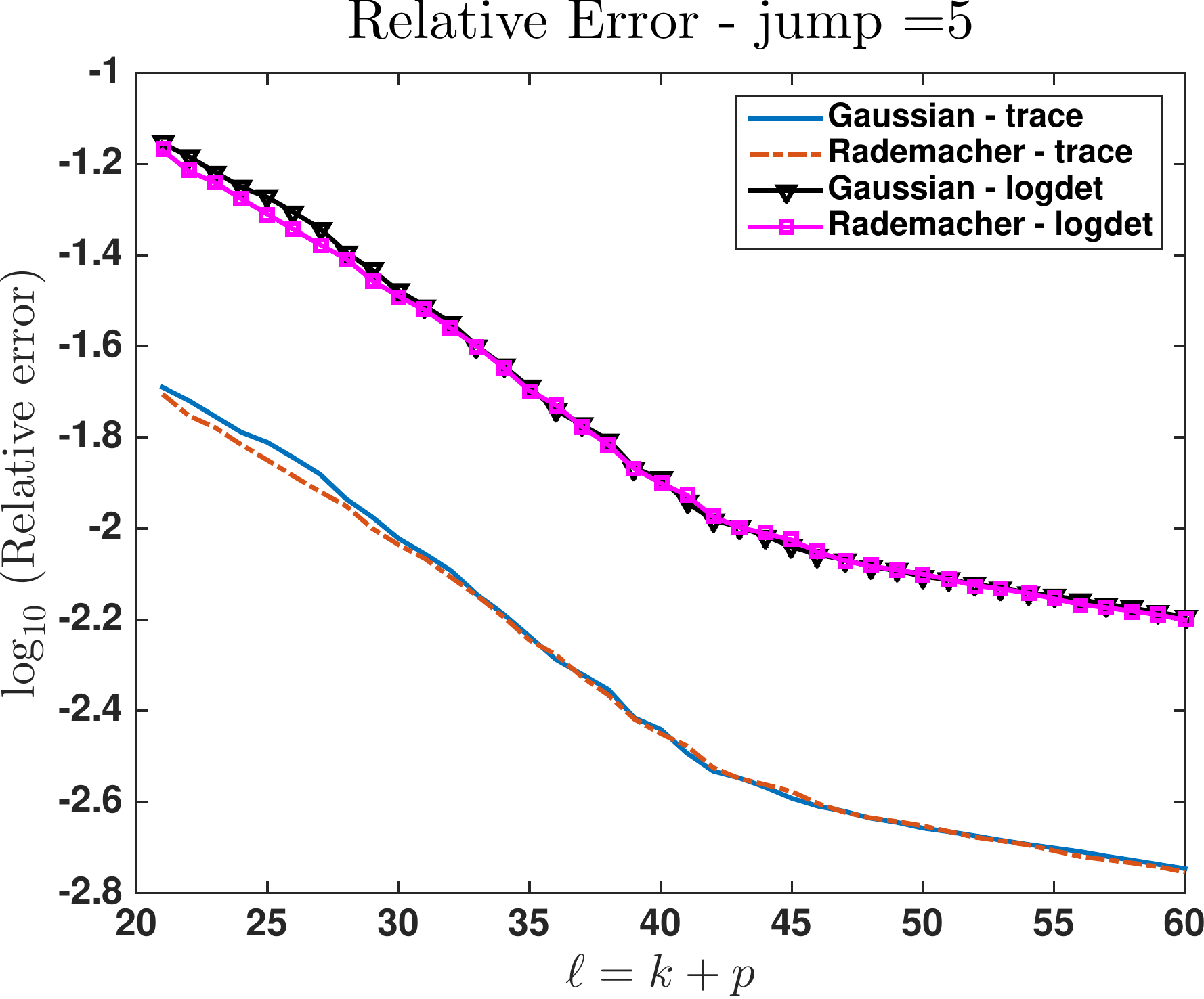}
\includegraphics[scale=0.3]{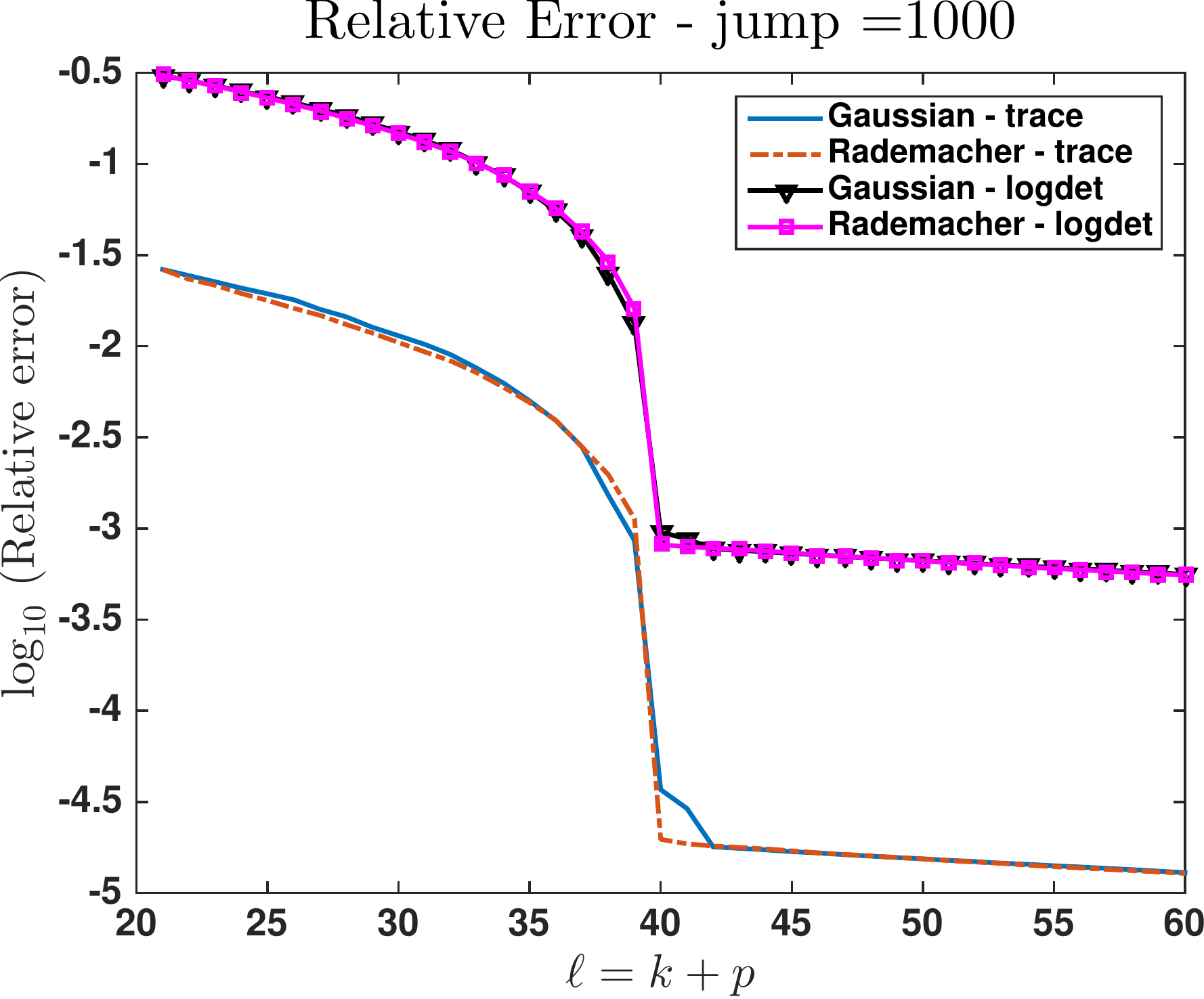}
\caption{Accuracy of trace and logdet computations for the matrix in~\eqref{e_medium}. (left) $l=1, h=5$ and (right) $l=1,h=1000$.}
\label{f_med1}
\end{figure}

First we fix $l=1$ and consider two different cases $h=5$, and $1000$. The oversampling parameter $p=20$ and the subspace iteration parameter is $q=1$. The results are displayed in Figure~\ref{f_med1}. The accuracy of both the trace and the logdet estimators improves considerably around the sample size $\ell = 40$ mark, when the eigenvalues undergo the large jump for $h=1000$; the transition is less sharp when $h=5$. This demonstrates the benefit of having a large eigenvalue gap in the accuracy of the estimators. As an extreme case, consider $l=0$ and $h=2$. In this example, the matrix $\ma$ has exactly rank $40$, and therefore $40$ matrix-vector products with $\ma$ are enough to recover the trace and logdet to machine precision. This result highlights the power of our proposed estimators. 
\begin{figure}[!ht]\centering
\includegraphics[scale=0.4]{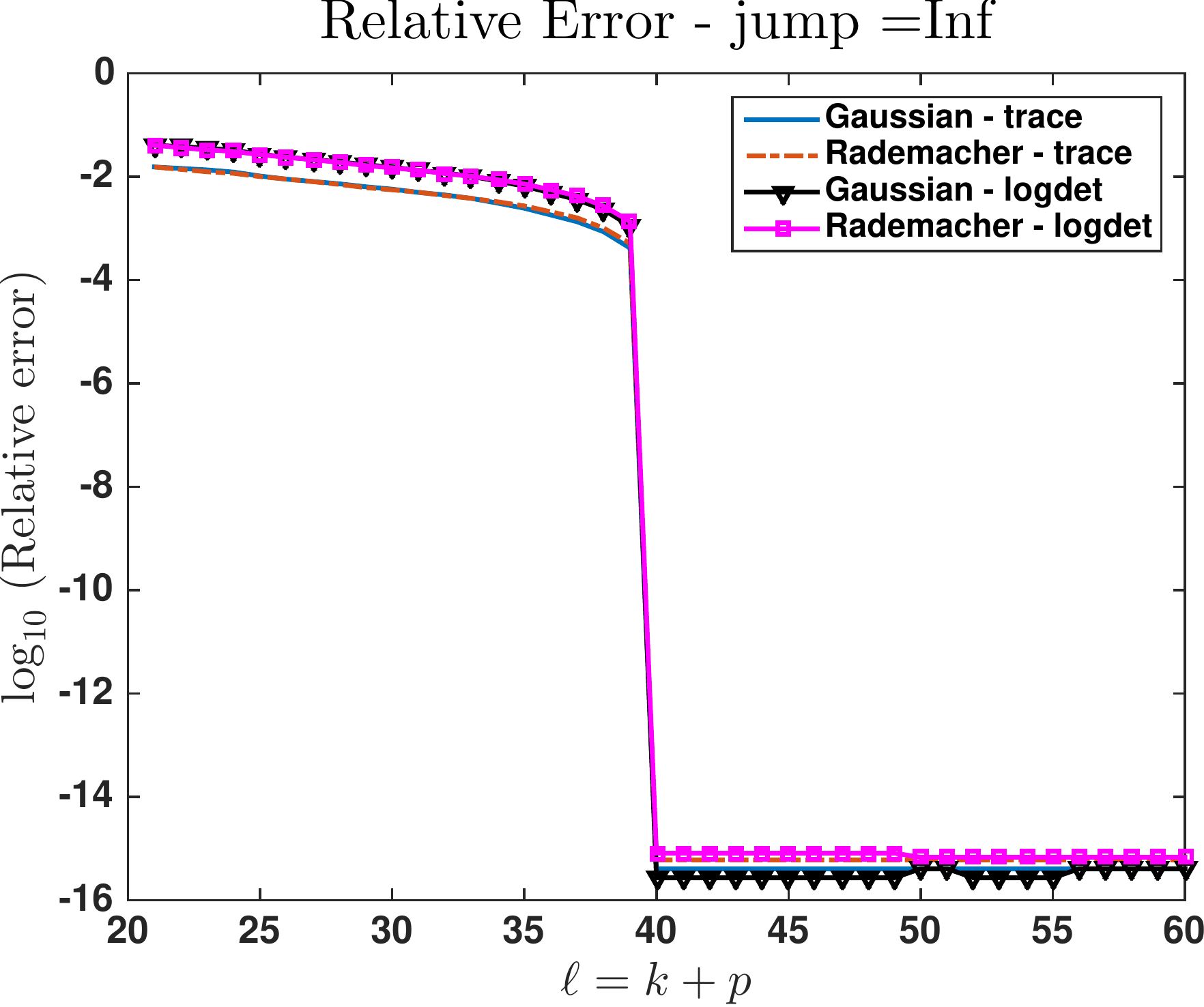}
\caption{Accuracy of trace and logdet computations for the matrix in~\eqref{e_medium} with $l=0,h=2$.}
\label{f_med2}
\end{figure}

\begin{figure}\centering
\begin{tabular}{cc}
\includegraphics[width=0.4\textwidth]{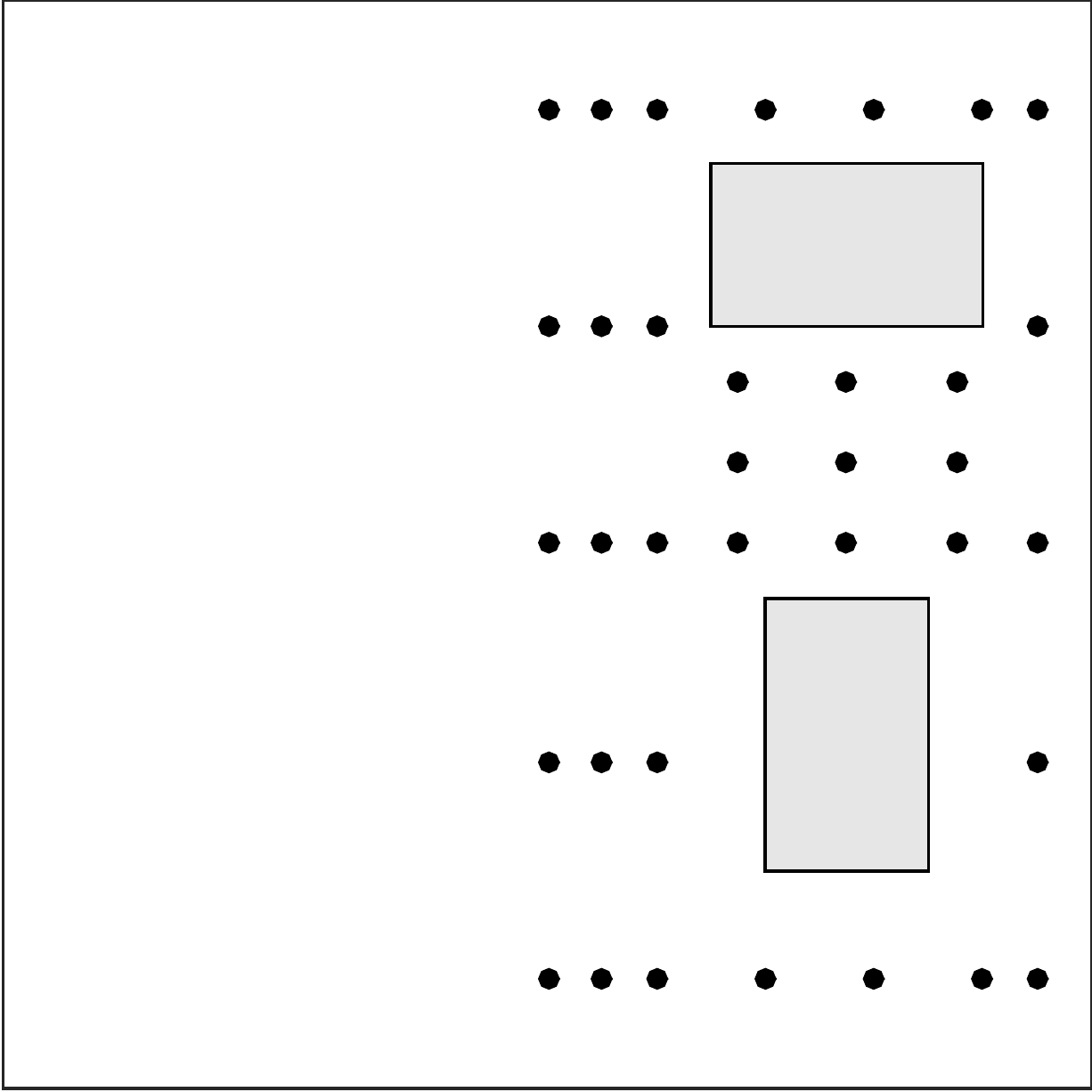}&
\includegraphics[width=0.4\textwidth]{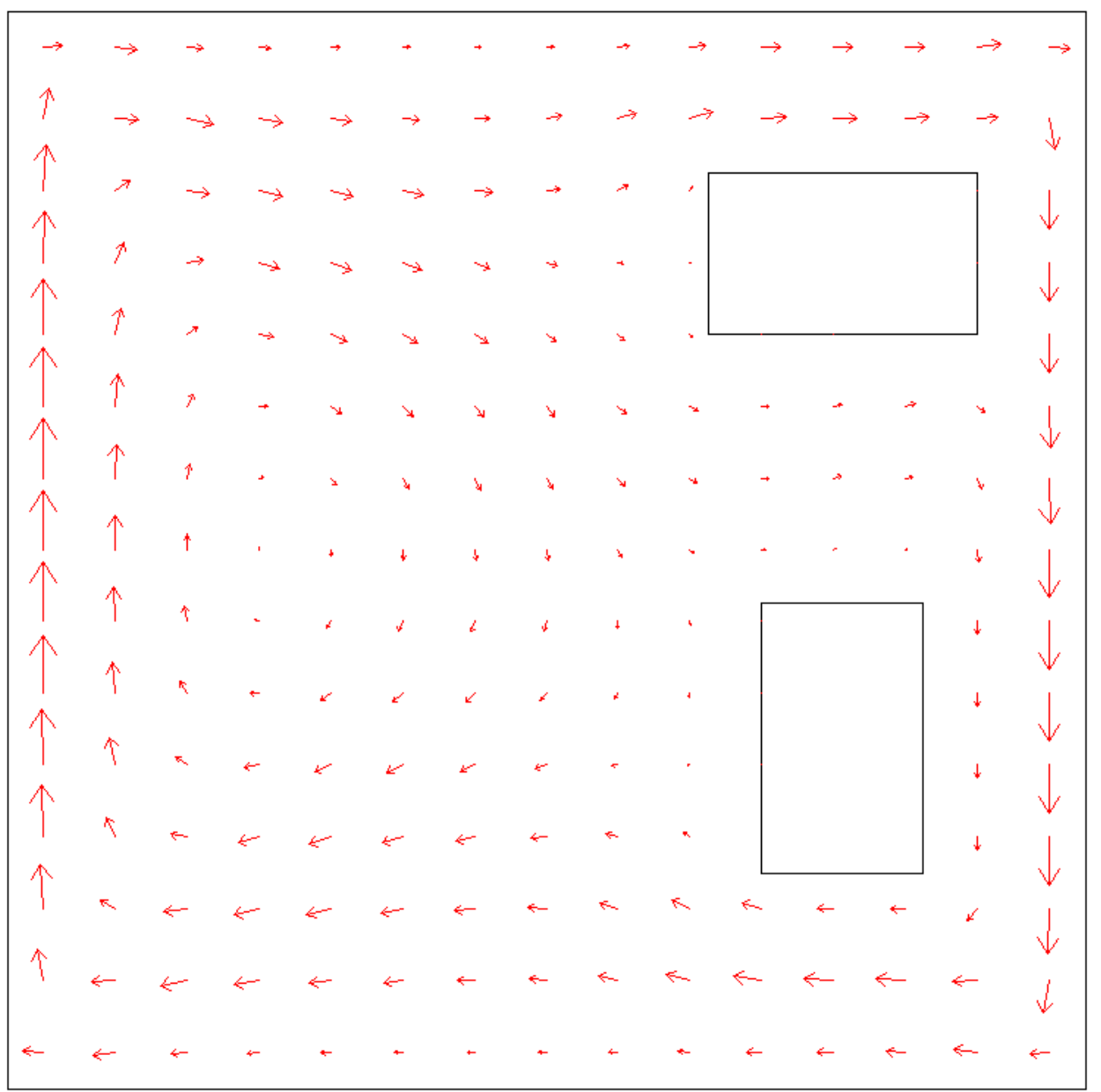} 
\end{tabular}
\caption{Left: the computational domain $\D$ 
is the region $[0,1]^2$ with two rectangular regions (representing buildings) removed.
The black dots indicate the locations of sensor where observations are recorded.
Right: the velocity field.} 
\label{fig:ad_dom}
\end{figure}

\section{Applications to evaluation of uncertainty quantification measures}\label{s_uq}
As mentioned in the introduction, the computation of
traces and log-determinants of high-dimensional operators is essential is the
emerging field of uncertainty quantification.  In this section, we use the
methods developed in this article to compute some common statistical quantities
that appear in the context of Bayesian inverse problems.  In particular, we
focus on a time-dependent advection-diffusion equation in which we seek to
infer an uncertain initial condition from measurements of the concentration at
discrete points in space/time; This is a commonly used example in the inverse
problem community; see e.g.,
\cite{AlexanderianPetraStadlerEtAl14,AkcelikBirosDraganescuEtAl05,FlathWilcoxAkcelikEtAl11,PetraStadler11}.
Below, we briefly outline the components of the Bayesian inverse problem. The
model problem used here is an adaptated from~\cite{AlexanderianPetraStadlerEtAl14}, and therefore, we refer the readers
to that paper for further details.

\paragraph{The forward problem} 
The forward problem models diffusive transport of a contaminant 
in a domain $\D \subset \mathbb{R}^2$, which is depicted in
Figure~\ref{fig:ad_dom} (left). The domain boundary $\partial \D$, is a combination of the outer edges of the domain 
as well as the internal boundaries of the rectangles
that model buildings.
The \emph{forward operator} maps an initial condition $\ipar$ 
to space/time observations of the contaminant concentration, by solving the
advection diffusion equation,
\begin{equation}\label{eq:ad}
  \begin{aligned}
    u_t - \kappa\Delta u + \mathbf{v}\cdot\nabla u &= 0 & \quad&\text{in
    }\D\times (0,T), \\
    u(\cdot, 0) &= \ipar  &&\text{in } \D , \\
    \kappa\nabla u\cdot \vec{n} &= 0 &&\text{on } \partial\D \times (0,T),
  \end{aligned}
\end{equation}
and extracting solution values at spatial points (sensor locations as indicated in Figure~\ref{fig:ad_dom} (left)) and at
pre-specified times. Here, $\kappa > 0$ is the diffusion coefficient and $T >
0$ is the final time.  In our numerical experiments, we use $\kappa = 0.001$.  
The velocity field $\vec{v}$, shown in Figures~\ref{fig:ad_dom}, is obtained by
solving a steady Navier-Stokes equation with the side walls driving the
flow; see~\cite{AlexanderianPetraStadlerEtAl14} for details.  The discretized
equations give rise to a discretized linear solution operator for the forward problem, 
which is composed with an observation operator to extract the space-time observations. 
We denote this discretized forward operator by $\mat{F}$. 

\paragraph{The Bayesian inverse problem}
The inverse problem aims to use a vector of observed data $\obs$, which
consists of sensor measurements at discrete points in time, to reconstruct the
uncertain initial condition. The dimension of $\obs$, which we denote by $q$,
is given by the product of the number of sensors and the number of points in
time where observations are recored. In the present example, we use $35$
sensors and record measurements at $t = 1$, $t = 2$, and $t = 3.5$. Therefore,
we have $\obs \in \mathbb{R}^{N_\text{obs}}$, with $N_\text{obs} = 105$, and that $\mat{F}:\mathbb{R}^n
\to \mathbb{R}^{105}$. We use a Gaussian prior measure
$\GM{\dpar_0}{\mat{C}_0}$, and use an
additive Gaussian noise model. Following~~\cite{AlexanderianPetraStadlerEtAl14}, the prior covariance 
is chosen to be the discretized biharmonic operator. 
The solution of the Bayesian inverse problem is
the posterior measure, $\GM{\dpar_\text{post}}{\mat{C}_\text{post}}$ with
\[
\mat{C}_\text{post} = (\mat{F}^* \boldsymbol{\Gamma}_\text{noise}^{-1} \mat{F} + \mat{C}_0^{-1})^{-1},
\qquad
\dpar_\text{post} = \mat{C}_\text{post}(\mat{F}^* \boldsymbol{\Gamma}_\text{noise}^{-1} \obs + \mat{C}_0^{-1}\dpar_0),
\]
We denote, by $\mat{H} \equiv \mat{F}^* \boldsymbol{\Gamma}_\text{noise}^{-1} \mat{F}$, the Fisher information matrix.
In many applications $\mat{H}$ has a rapidly decaying spectrum; see Figure~\ref{fig:hessian} (left). Moreover,
in the present setup, the rank of $\mat{H}$ is bounded by the dimension of the observations, which in our example is given by $N_\text{obs} = 105$.
The prior-preconditioned Fisher information matrix 
\[
   \pH = \mat{C}_0^{1/2} \mat{H} \mat{C}_0^{1/2}, 
\] 
is also of importance in what follows. Notice that preconditioning of $\mat{H}$
by the prior, due to the smoothing properties of the priors employed in the
present example, results in a more rapid spectral decay; see
Figure~\ref{fig:hessian}(right). 

 We point out that the quantity $\trace(\pH)$
is related to the \emph{sensitivity criterion} in optimal experimental design (OED) theory~\cite{Ucinski05}.
On the other hand, $\log\det(\mat{I} + \pH)$ is related to Bayesian D-optimal
design criterion~\cite{ChalonerVerdinelli95}.  As shown
in~\cite{AlexanderianGloorGhattas15}, $\log\det(\mat{I} + \pH)$ is the expected
information gain from the posterior measure to the prior measure in a Bayesian
linear inverse problem with Gaussian prior and noise distributions, and with an
inversion parameter that belongs to a Hilbert space. Note that in the present context, 
information gain is quantified by the Kullback--Leibler divergence from posterior measure 
to prior measure. A detailed discussion of uncertainty measures is also provided in~\cite{SaibabaKitanidis15}.

In Figure~\ref{fig:error_estimates}, we report the error in approximation of
$\trace(\mh)$ and $\log\det(\mat{I} + \pH)$.  Both of these quantities are of
interest in theory OED, where one is
interested in measures of uncertainty in reconstructed
parameters~\cite{AtkinsonDonev92,Ucinski05}.  Such statistical measures are
then used to guide the experimental configurations used to collect experimental
data so as to maximize the statistical quality of the reconstructed/inferred parameters. 
Note that, in the present example, an experimental configuration is given by the
placement of sensors (black dots in Figure~\ref{fig:ad_dom} (left) where
concentration data is recorded). 
\begin{figure}\centering
\begin{tabular}{cc}
\includegraphics[width=0.4\textwidth]{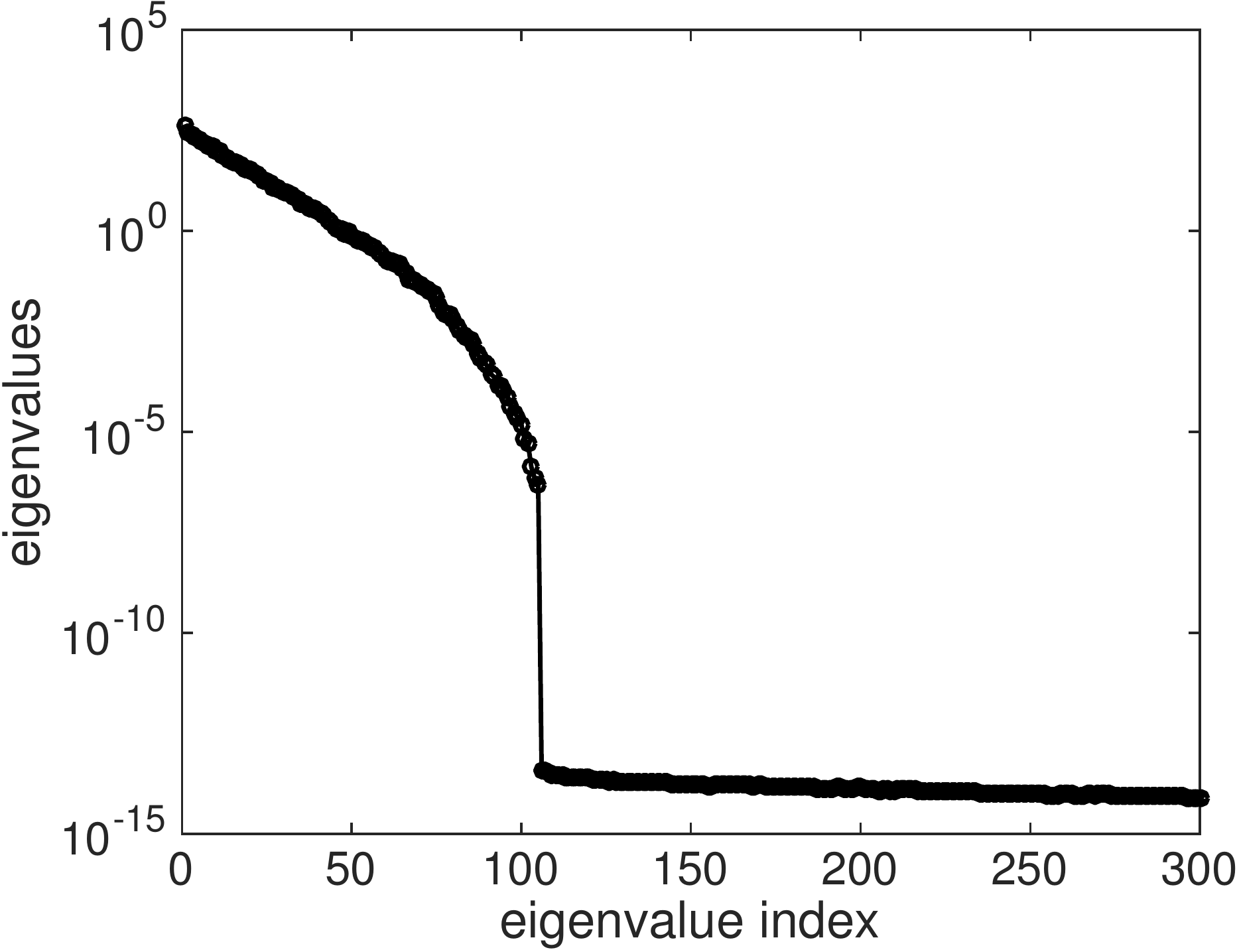}
&
\includegraphics[width=0.4\textwidth]{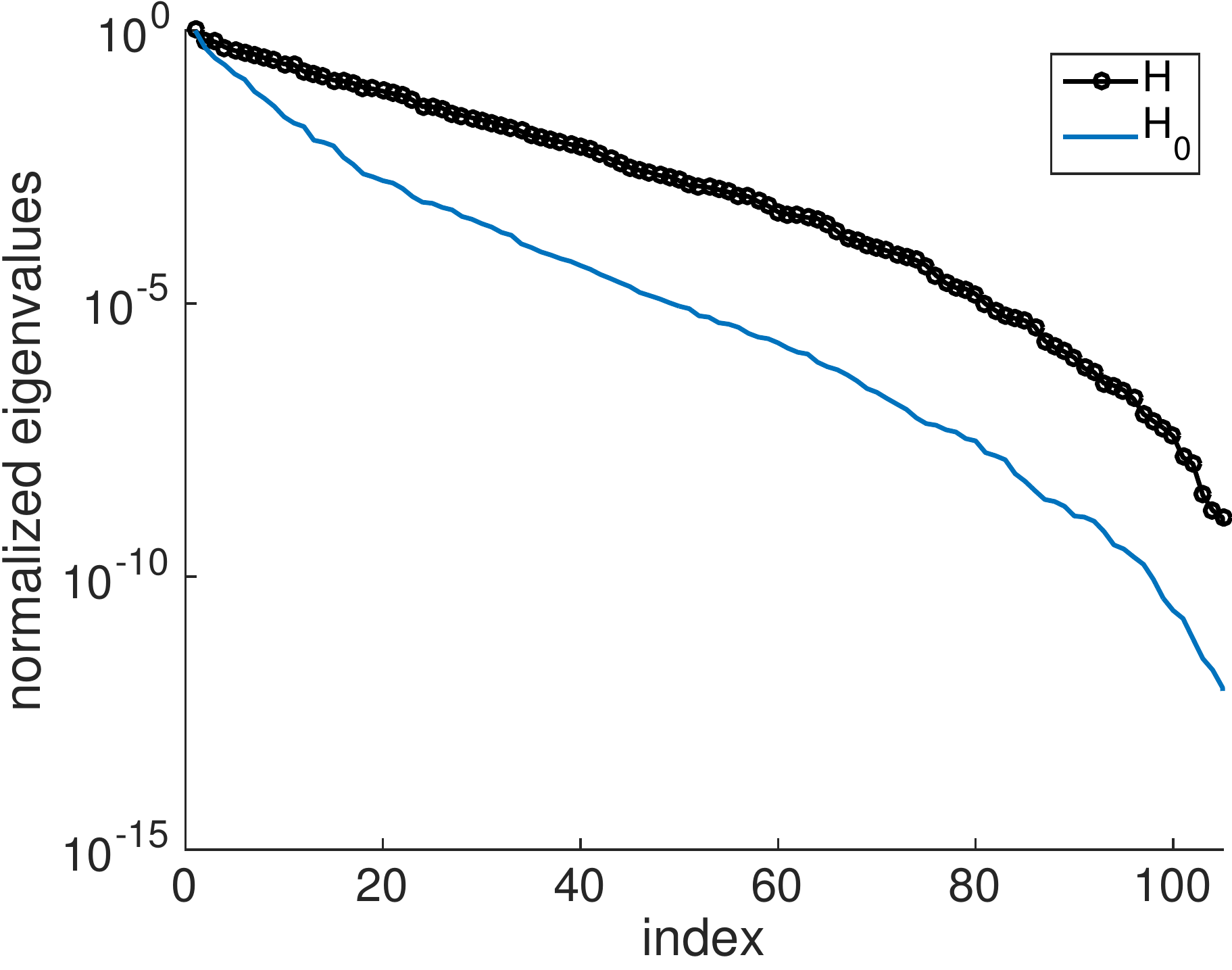}
\end{tabular}
\caption{Left: first $300$ eigenvalues of $\mat{H}$; right: normalized nonzero eigenvalues of $\mat{H}$ and $\pH$.}
\label{fig:hessian}
\end{figure}

\begin{figure}\centering
\begin{tabular}{cc}
\includegraphics[width=0.4\textwidth]{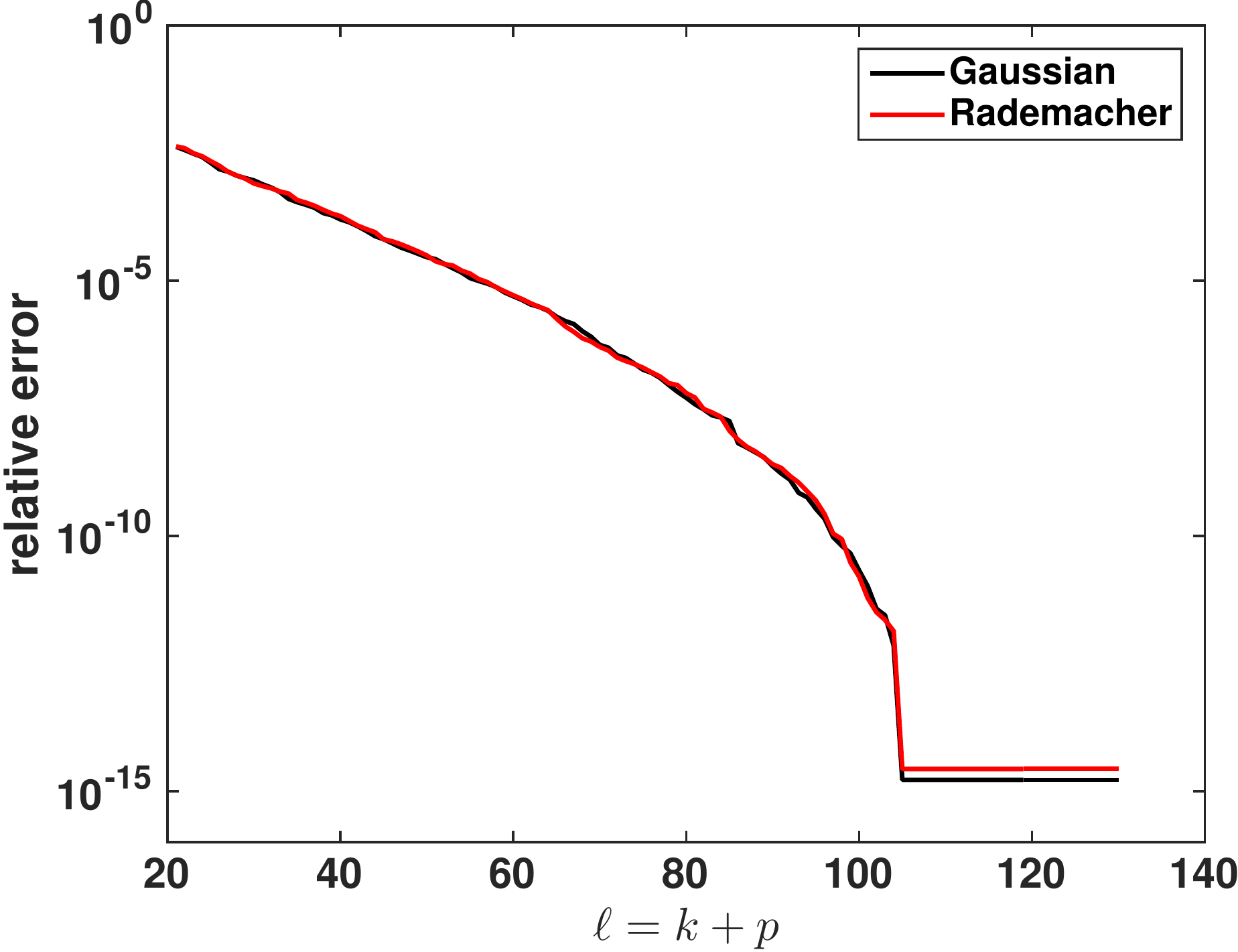}
&
\includegraphics[width=0.4\textwidth]{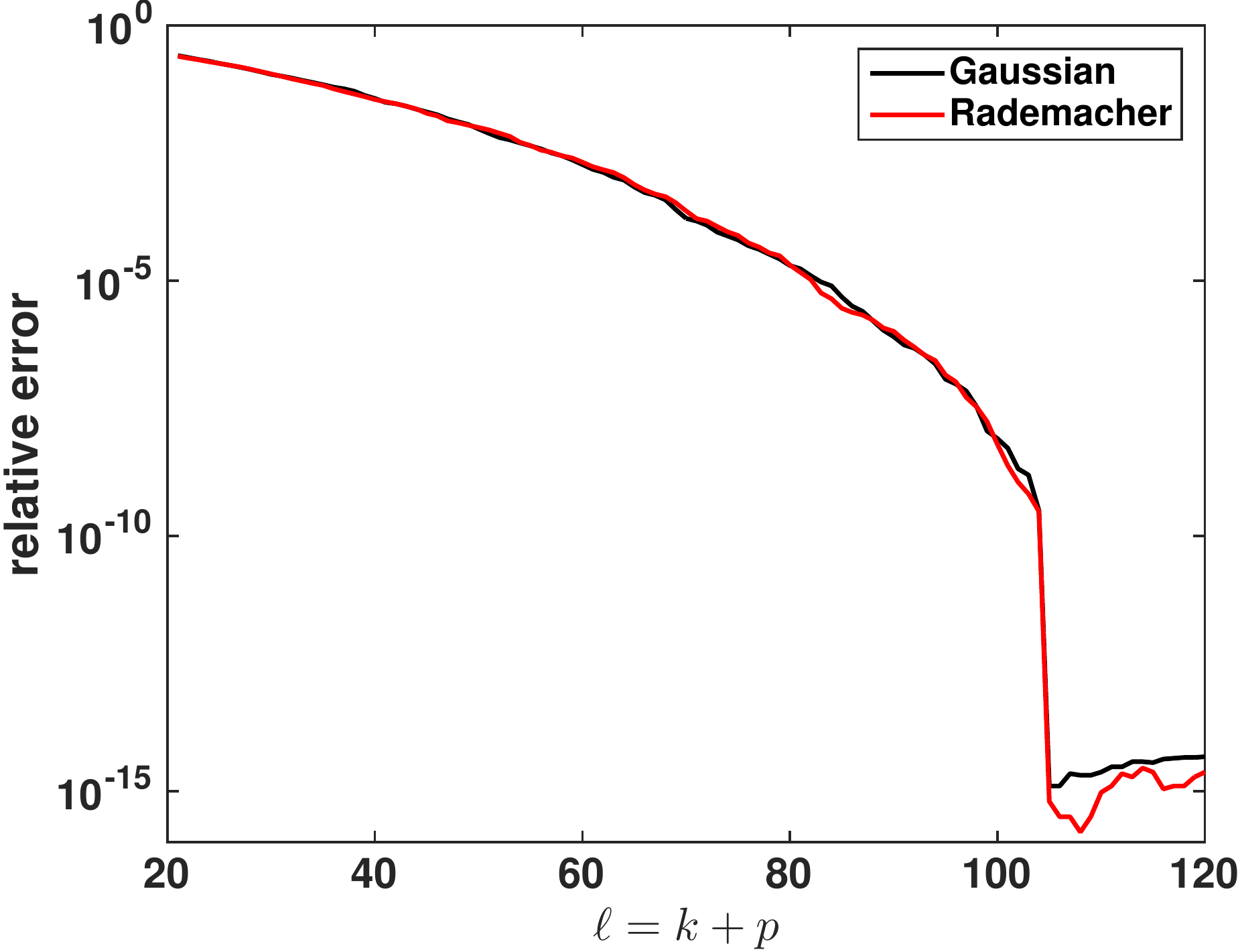}
\end{tabular}
\caption{Left: error in estimation of $trace(\pH)$; right: error in estimation of
$\log \det( \pH + \mat{I} )$. Computations were done using $p = 20$ and the randomized estimators use
$\ell = k + p$ random vectors with increasing values of $k$.}
\label{fig:error_estimates}
\end{figure}

\section{Conclusion}
We present randomized estimators for the
trace and log-determinant of implicitly defined Hermitian positive semi-definite matrices.
The estimators are low-rank approximations computed with subspace iteration.
We show, theoretically and numerically, that our estimators are effective for matrices
with a large eigenvalue gap or rapidly decaying eigenvalues.

Our error analyses for the estimators are cleanly separated 
into two parts: A structural analysis, which is applicable to any choice of a starting guess, paves
the way for a probabilistic analysis, in this case for Gaussian and
Rademacher starting guesses.  In addition, we derive asymptotic bounds on the
number of random vectors required to guarantee a specified accuracy with low
probability of failure.  We present comprehensive numerical experiments to
illustrate the performance of the estimators, and demonstrate their suitability 
for challenging application problems, such as the computation of the expected information
gain in a Bayesian linear inverse problem governed by a time-dependent PDE.

Future work will evolve around two main issues.

\paragraph{Rademacher random matrices.}
Our analysis implies that 
a Gaussian starting guess can do with a fixed oversampling parameter, 
while the oversampling amount for a Rademacher starting guess depends
on the dimension of the dominant eigenspace and the dimension of the matrix.  
However, the numerical experiments indicate that, for both types of starting guesses, 
an oversampling parameter of~$20$ leads to accurate  estimators.
We plan to further investigate estimators with 
Rademacher starting guesses, and specifically to derive  error bounds for 
the expectation of the corresponding estimators.
Another issue to be explored is the tightness of the bound $\ell \sim (k+\log n)\log k$ 
for Rademacher starting guesses.

\paragraph{Applications.}
We plan to integrate our estimators into computational methods
for large-scale uncertainty quantification. Our main goal is the
computation of optimal experimental designs (OED) for large-scale inverse
problems. This can be posed as an optimization problem, where the objective function is the trace 
or log-determinant of a high-dimensional operator. 
Due to their efficiency and high accuracy, we expect that our estimators
are well suited for OED.

\appendix
\section{Gaussian Random Matrices}\label{app_gauss}
In this section, we state a lemma on the pseudo-inverse of a rectangular Gaussian random matrix, and use this result to prove both parts of Lemma~\ref{p_2norm2}. 
\subsection{Pseudo-inverse of a Gaussian random matrix}
We state a result on the large deviation bound of the pseudo-inverse of a Gaussian random matrix~\cite[Proposition 10.4]{HMT09}.

\begin{lemma}\label{l_gauss_pinv} Let $\mg \in \mathbb{R}^{k\times(k+p)}$ be a random Gaussian matrix and let $p\geq 2$. For all $t \geq 1$, 
\begin{equation} 
\prob{\normtwo{\mg^\dagger} \geq \frac{e\sqrt{k+p}}{p+1}\cdot t} \leq t^{-(p+1)}. 
\end{equation}
\end{lemma}
\subsection{Proof of Lemma~\ref{p_2norm2}}
\begin{proof}
 From~\cite[Corollary 5.35]{vershynin2012introduction} we have 
\[ \prob{\normtwo{\mg_2} > \sqrt{n-k} + \sqrt{k+p} + t} \leq \exp(-t^2/2). \]
Recall from~\eqref{e_mu} $\mu = \sqrt{n-k} + \sqrt{k+p}$. From the \textit{law of the unconscious statistician}~\cite[Proposition S4.2]{gu2015subspace}, 
\begin{align*}
\expect{\normtwo{\mg_2}^2}  = &\>  \int_0^\infty 2t \prob{\normtwo{\mg_2} > t} dt\\
\leq & \> \int_0^{\mu} 2t dt + \int_{\mu}^\infty  2t \prob{\normtwo{\mg_2} > t} dt \\
\leq & \> \mu^2 + \int_{0}^\infty 2(u+\mu) \exp(-u^2/2)du 
=  \> \mu^2 + 2\left(1+ \mu\sqrt{\frac{\pi}{2}}\right).
\end{align*}•
This concludes the proof for~\eqref{2norm2}. 

Next consider~\eqref{pinv2norm2}. Using Lemma~\ref{l_gauss_pinv}, we have for $t  > 0$ 
\begin{equation}
\prob{\normtwo{\mg_1^\dagger} \geq t} \leq\> D t^{-(p+1)} \qquad D \define \> \frac{1}{\sqrt{2\pi(p+1)}}\left( \frac{e\sqrt{k+p}}{p+1}\right). 
\end{equation}
As before, we have 
\begin{align*}
\expect{\normtwo{\mg_1^{\dagger}}^2}  = &\quad  \int_0^\infty 2t \prob{\normtwo{\mg_1^\dagger} > t} dt\\
\leq & \quad  \int_0^{\beta} 2t dt + \int_{\beta}^\infty  2t \prob{\normtwo{\mg_1^\dagger} > t} dt\\\
\leq & \quad \beta^2 + \int_{\beta}^\infty  2t D t^{-(p+1)} dt = \beta^2 + 2D\frac{\beta^{1-p}}{p-1}.
 \end{align*}•
Minimizing w.r.t. $\beta$, we get $\beta = (D)^{1/(p+1)}$. Substitute this value for $\beta$ and simplify. \qed
\end{proof}

\section{Rademacher Random Matrices}~\label{app_rad}
In this section, we state the matrix Chernoff inequalities~\cite{TroppBook} and other useful 
concentration inequalities and use these results to prove Theorem~\ref{t_rad}.
\subsection{Useful concentration inequalities}
The proof of Theorem~\ref{t_rad} relies on the matrix concentration inequalities developed in~\cite{TroppBook}. 
We will need the following result~\cite[Theorem 5.1.1]{TroppBook} in what follows.
\begin{theorem}[Matrix Chernoff]\label{t_mat_chern}
Let $\{\mx_k\}$ be  finite sequence of independent, random, $d\times d$ Hermitian matrices. Assume that $0 \leq \lambda_{\min}(\mx_k)$ and $\lambda_{\max}(\mx_k) \leq L$ for each index $k$. Let us define 
\[\mu_{\min} \equiv \lambda_{\min}\left(\sum_k\expect{\mx_k} \right) \qquad \mu_{\max} \equiv \lambda_{\max}\left(\sum_k\expect{\mx_k} \right),\]
and let $g(x) \equiv e^x(1+x)^{-(1+x)}.$ Then for any $\epsilon > 0$
 \[ \prob{\lambda_{\max}\left(\sum_k \mx_k \right) \geq (1+\epsilon)\mu_{\max}} \leq d g(\epsilon)^{\mu_{\max}/L}, \]
and for any $0 \leq \epsilon <1 $
\[ \prob{\lambda_{\min}\left(\sum_k \mx_k \right) \leq (1-\epsilon)\mu_{\min}} \leq d g(-\epsilon)^{\mu_{\min}/L}. \]
\end{theorem}

The following result was first proved by Ledoux~\cite{ledoux1997talagrand} but we reproduce the statement from~\cite[Proposition 2.1]{tropp2011improved}.
\begin{lemma}\label{l_rad_func} Suppose $f:\mathbb{R}^n \to \mathbb{R}$ is a 
convex function that satisfies the following Lipschitz bound 
\[ |f(\vx)-f(\vy)| \leq L \normtwo{\vx-\vy} \qquad \text{for all }\>\vx,\vy \in \mathbb{R}^n. \]
Let $\vz \in \mathbb{R}^n$ be a random vector with entries drawn from an i.i.d.\ Rademacher distribution. 
Then, for all $t \geq 0$, 
\[
   \prob{f(\vz) \geq \expect{f(\vz)} + Lt } \leq e^{-t^2/8}.
\]
\end{lemma}

\begin{lemma}\label{l_rad_vtw} Let $\mv$ be a $n\times r$ matrix with orthonormal columns and let $n\geq r$. Let $\vz $ be an $n\times 1$ vector with entries drawn from an i.i.d.\ Rademacher distribution. Then, for $0 < \delta < 1$, 
\[ \prob{\normtwo{\mv^*\vz} \geq \sqrt{r} + \sqrt{8\log\left(\frac{1}{\delta}\right)}} \leq \delta. \] 
\end{lemma}
 \begin{proof}
Our proof follows the strategy in~\cite[Lemma 3.3]{tropp2011improved}. 
Define the function $f(\vx) = \normtwo{\mv^*\vx}$. We observe that $f$ satisfies the
assumptions of Lemma~\ref{l_rad_func}, with Lipschitz constant $L=1$; the latter follows from 
\[|\normtwo{\mv^*\vx}-\normtwo{\mv^*\vy}| \leq \normtwo{\mv^*(\vx-\vy)} \leq \normtwo{\vx-\vy}. \]
Furthermore, using H\"{o}lder's inequality 
\[ \expect{f(\vz)} \leq [\expect{f(\vz)^2}]^{1/2} = \| \mv\|_F = \sqrt{r}.\]
Using Lemma~\ref{l_rad_func} with $t_\delta = \sqrt{8\log\left(1/\delta\right)}$ 
we have 
\[
   \prob{f(\vz) \geq \sqrt{r} + t_\delta} \leq \prob{f(\vz) \geq \expect{f(\vz)} + t_\delta} 
\leq e^{-t_\delta^2/8} = \delta.~\qed
\]
\end{proof}
\begin{lemma}\label{l_max}
Let $X_i$ for $i=1,\dots,n$ be a sequence of i.i.d.\ random variables. If for each $i=1,\dots,n$,  $\prob{X_i \geq a} \leq \xi$ holds, where $\xi \in (0,1]$, then 
\[ \prob{\max_{i=1,\cdots,n} X_i \geq a} \leq n\xi. \]
\end{lemma}
\begin{proof}
 Since $ \prob{X_i \geq a} \leq \xi$ then $\prob{X_i < a} \geq 1 - \xi $. We can bound
\begin{align*}
\prob{\max_{i=1,\cdots,n} X_i \geq a}  = & \left(1-\prob{\max_{i=1,\cdots,n} X_i < a}\right) \\
= & \left(1- \prod_{i=1}^n\prob{ X_i <a}\right) \leq 1 - (1-\xi)^n. 
\end{align*} 
The proof follows from Bernoulli's inequality~\cite[Theorem 5.1]{stirling2009mathematical} which states $(1-\xi)^n \geq 1 - n\xi$ for $\xi \in [0,1]$ and $n\geq 1$.~\qed
\end{proof}
\subsection{Proof of Theorem~\ref{t_rad}}
\begin{proof}
Recall that $\mom_1 = \mU^*_1 \mom$ and $\mom_2 = \mU_2^*\mom$ where $\mom$ is random matrix with entries chosen from an i.i.d.\ Rademacher distribution. The proof proceeds in three steps.  
\paragraph{1. Bound for $\normtwo{\mom_2}^2$}
The proof uses the matrix Chernoff concentration inequality. Let $\vomega_i \in \mathbb{R}^{n\times 1}$ be the $i$-th column of $\mom$. Note $  \mom_2\mom_2^* \in \mathbb{C}^{(n-k)\times (n-k)}$ and 
\[ \expect{\mom_2\mom_2^*} = \sum_{i=1}^\ell \mU_2^*\expect{\vomega_i\vomega_i^*}\mU_2 = \ell \mi_{n-k}. \]
Furthermore, define $\mu_\text{min} ( \mom_2\mom_2^*) \equiv \lambda_\text{min} (\expect{ \mom_2\mom_2^*})$ and $\mu_\text{max} ( \mom_2\mom_2^*) \equiv \lambda_\text{max} (\expect{ \mom_2\mom_2^*})$. Clearly $\mu_\text{min} = \mu_\text{max} = \ell $. 
Note that here we have expressed $ \mom_2\mom_2^*$ as a finite sum of $\ell$ rank-1 matrices, each with a single nonzero  eigenvalue $\vomega_i^* \mU_2\mU_2^*\vomega_i $. We want to obtain a probabilistic bound for the maximum eigenvalue i.e., $L_2 = \max_{i=1,\cdots,\ell} \normtwo{\mU_2^*\vomega_i}^2$. Using Lemma~\ref{l_rad_vtw} we can write with probability at most $e^{-t^2/8}$ 
\[ \left(\sqrt{n-k} + t\right)^2 \leq \normtwo{\mU_2^*\vomega_i}^2  =  \vomega_i^* \mU_2\mU_2^*\vomega_i .   \]

Since $\normtwo{\mU_2^*\vomega_i}^2$ are i.i.d., applying Lemma~\ref{l_max} gives 
\begin{equation*}\prob{\max_{i=1,\cdots,\ell} \normtwo{\mU_2^*\vomega_i} \geq \sqrt{n-k} + t } \leq \ell e^{-t^2/8}.\end{equation*}
Take $t = \sqrt{ 8\log(4\ell/\delta)}$ to obtain %
\begin{equation}\label{e_l2} \prob{L_2 \geq C_u^2} \leq \delta/4, \qquad C_u \equiv \sqrt{n-k} + \sqrt{ 8\log\left(\frac{4\ell}{\delta}\right)}. 
\end{equation}

The matrix $\mom_2$ satisfies the conditions of the matrix Chernoff theorem~\ref{t_mat_chern}; for $\eta \geq 0$ we have 
\[ \prob{\lambda_\text{max}(\mom_2\mom_2^*)  \geq (1+\eta) \ell } \leq (n-k) g(\eta)^\frac{\ell}{L_2},  \]
where the function $g(\eta)$ is defined in Theorem~\ref{t_mat_chern}. For $\eta > 1$ the Chernoff bounds can be simplified~\cite[Section 4.3]{mitzenmacher2005probability} since $g(\eta) \leq e^{-\eta/3}$, to obtain 
\[ \prob{\lambda_\text{max}(\mom_2\mom_2^*)  \geq (1+\eta) \ell } \leq (n-k) \exp\left(-\frac{\eta\ell}{3L_2}\right). \]

Choose the parameter $$\eta_\delta = C_{\ell,\delta}C_u^2 = 3\ell^{-1}C_u^2\log\left(\frac{4(n-k)}{\delta}\right),$$
so that 
\begin{align*}
\prob{\normtwo{\mom_2}^2 \geq (1+\eta_\delta)\ell} \leq & \> (n-k)\exp\left(-\frac{C_u^2}{L_2}\log\frac{4(n-k)}{\delta}\right) \\
= & \> (n-k)\left(\frac{\delta}{4(n-k)}\right)^{C_u^2/L_2}. 
\end{align*}

Finally, we want to find a lower bound for $\normtwo{\mom_2}^2$. Define the events 
\[
 A =  \>\left\{\mom_2 \mid L_2 <   C_u^2 \right\}, \qquad 
 B =  \>\left\{ \mom_2 \mid \normtwo{\mom_2}^2 \geq (1+\eta_\delta)\ell \right\}. 
\]
Note that $\prob{A^c} \leq \delta/4$ and under event $A$ we have  $C_u^2 >L_2$ so that 
$$\prob{B\mid A}  \leq (n-k)\left(\frac{\delta}{4(n-k)}\right)^{C_u^2/L_2} \leq  \delta/4.$$  
Using the law of total probability 
\begin{align*}\prob{B} =& \> \prob{B\mid A}\prob{A} + \prob{B\mid A^c} \prob{A^c} \\
\leq & \> \prob{B\mid A} + \prob{A^c}, 
\end{align*} 
we can obtain a bound for $\prob{B}$ as
\[ \prob{\normtwo{\mom_2}^2 \geq \ell\left(1 + C_u^2 C_{\ell,\delta}\right)  } \leq \delta/2.\]

\paragraph{2. Bound for $\normtwo{\mom_1^\dagger}^2$} 
The steps are similar and we again use the matrix Chernoff concentration
inequality. Consider $\mom_1\mom_1^*\in \mathbb{C}^{k\times k}$, and
as before, write this matrix as the sum of rank-1 matrices to obtain 
\[ 
\expect{\mom_1\mom_1^*} = \sum_{i=1}^\ell
\mU_1^*\expect{\vomega_i\vomega_i^{*}}\mU_1 = \ell \mi_{k}, \] and
$\mu_\text{min} (\mom_1\mom_1^*) = \ell$. Each summand in the above decomposition of 
$\mom_1\mom_1^*$ has one nonzero eigenvalue
$\vomega_i^* \mU_1\mU_1^*\vomega_i $. Following the same strategy as in Step 1, we define $L_1 \equiv
\max_{i=1,\dots,\ell}\normtwo{\mU_1^*\vomega_i}^2 $ and apply
Lemma~\ref{l_rad_vtw} to obtain \begin{equation*} \prob{\max_{i=1,\cdots,\ell}
\normtwo{\mU_1^*\vomega_i} \geq \sqrt{k} + t } \leq \ell e^{-t^2/8} \leq
ne^{-t^2/8}.\end{equation*}

Take $t = \sqrt{ 8\log(4n/\delta)}$ to obtain   
\begin{equation}\label{e_l1} 
\prob{ L_1 \geq C_l^2} \leq \delta/4, \qquad C_l \equiv \sqrt{k} + \sqrt{ 8\log\left(\frac{4n}{\delta}\right)}. 
\end{equation}

A straightforward application of the Chernoff bound in Theorem~\ref{t_mat_chern} gives us 

\[ {\prob{\lambda_\text{min}(\mom_1\mom_1^*)  \leq (1-\rho) \ell } \leq k g(-\rho)^\frac{\ell}{L_1}.  }\]
 Next, observe that $-\log g(-\rho)$ has the Taylor series expansion in the region $0 < \rho < 1$ 
\[ -\log g(-\rho) = {\rho} + (1-\rho)\log(1-\rho)  = \frac{\rho^2}{2} +  \frac{\rho^3}{6} +  \frac{\rho^4}{12} + \dots   \]
so that $-\log g(-\rho) \geq \rho^2/2$ for $0 < \rho < 1$ or $g(-\rho) \leq e^{-\rho^2/2}$. This gives us
\begin{equation}\label{e_chern2} \prob{\normtwo{\mom_1^\dagger}^2  \geq  \frac{1}{ (1-\rho) \ell } }\>  \leq \> k  \exp\left(-\frac{\rho^2\ell}{2L_1}\right) , \end{equation}
where we have used $\lambda_\text{min}(\mom_1\mom_1^*) = 1/\normtwo{\mom_1^\dagger}^2$ assuming $\rank(\mom_1) = k$.

With  the number of samples as defined in Theorem~\ref{t_rademacher} 
\[ \ell \geq 2\rho^{-2}C_l^2 \log\left(\frac{4k}{\delta}\right),\] 
the Chernoff bound~\eqref{e_chern2} becomes 
\[ \prob{\normtwo{\mom_1^\dagger}^2  \geq  \frac{1}{ (1-\rho) \ell } }\>  \leq k\left(\frac{\delta}{4k}\right)^{C_l^2/L_1}.\] 
Define the events 
\[
C =  \left\{ \mom_1 \mid \normtwo{\mom_1^\dagger}^2  \geq  \frac{1}{ (1-\rho) \ell }\right\}, 
\qquad
D =  \{ \mom_1 \mid L_1 < C_\ell^2\}. 
\]
Note that $\prob{D^c} \leq \delta/4$ from~\eqref{e_l1}. Then since the exponent is strictly greater than $1$, we have 
\[ \prob{C \mid D }\>  \leq k\left(\frac{\delta}{4k}\right)^{C_l^2/L_1} \leq \delta/4.\]
Using the conditioning argument as before gives $\prob{C} \leq \delta/2$.

\paragraph{3. Combining bounds}  Define the event
\[
E =  \left\{ \mom \mid \normtwo{\mom_1^\dagger}^2 \geq \frac{1}{(1-\rho)\ell}\right\}, \qquad 
F =  \left\{ \mom \mid \normtwo{\mom_2}^2 \geq (1 + C_{\ell,\delta}C_u^2)\ell \right\},
\]
where $C_{\ell,\delta} $ is defined  in Step 1,  $\prob{E} \leq \delta/2$ and from Step 2, $\prob{F} \leq \delta/2$. It can be verified that
\[ \left\{ \mom \mid \normtwo{\mom_2}^2 \normtwo{\mom_1^\dagger}^2\geq \frac{1}{1-\rho}(1 + C_{\ell,\delta}C_u^2) \right\} \subseteq E \cup F,\]
and therefore, we can use the union bound 
\[ \prob{  \normtwo{\mom_2}^2 \normtwo{\mom_1^\dagger}^2 \geq \frac{1}{1-\rho}(1 + C_{\ell,\delta}C_u^2)  } \leq  \prob{E} + \prob{F} \leq \delta. \]  
Plugging in the value of $C_{\ell,\delta}$ and $C_u^2$ from Step 1 gives the desired result.~\qed 
\end{proof}

\bibliography{DIMLS}
\end{document}